\documentclass[12pt]{amsart}
\usepackage[utf8]{inputenc}
\usepackage{graphicx}

\usepackage[table,xcdraw]{xcolor}
\usepackage{multirow}

\usepackage{tikzit}

\tikzstyle{node}=[fill={rgb,255: red,64; green,64; blue,64}, draw=black, shape=circle, tikzit fill={rgb,255: red,64; green,64; blue,64}, tikzit draw=black]
\tikzstyle{rednode}=[fill=red, draw=red, shape=circle, tikzit fill=red, tikzit draw=red]
\tikzstyle{node-lite}=[fill={rgb,255: red,128; green,128; blue,128}, draw={rgb,255: red,128; green,128; blue,128}, shape=circle]
\tikzstyle{number}=[fill=white, draw=white, shape=circle, tikzit fill=white]

\tikzstyle{edge}=[-, draw={rgb,255: red,64; green,64; blue,64}, tikzit draw={rgb,255: red,64; green,64; blue,64}]

\usepackage{fourier}
\usepackage{array}
\usepackage{makecell}

\usepackage{ulem} 
\usepackage[breaklinks]{hyperref} 
\usepackage{microtype}
\usepackage{mathtools}
\usepackage{amsthm} 
\usepackage{amssymb} 
\usepackage{enumerate} 

\usepackage{xcolor,cancel}

\DeclareMathOperator*{\Plus}{\scalerel*{\color{gray} +}{\textstyle\sum}}
\DeclareMathOperator*{\Bullet}{\scalerel*{\bullet}{\textstyle\sum}}
\usepackage{scalerel}

\usepackage{fullpage}

\newcommand{\inv}{{^{-1}}} 

\newcommand{\ds}{\displaystyle}

\newcommand{\fa}{\text{ for all }} 
\newcommand{\abs}[1]{{\left |#1\right |}}  

\theoremstyle{plain} 
\newtheorem{theorem}{Theorem}[section]
\newtheorem{lemma}[theorem]{Lemma}
\newtheorem{corollary}{Corollary}[theorem]
\newtheorem{prop}{Proposition}[section]

\theoremstyle{definition} 
\newtheorem{definition}{Definition}[section]

\theoremstyle{remark} 
\newtheorem{remark}{Remark}

\newtheorem{example}{Example}[section]

\begin{document}

\title{Permutations Avoiding Certain Partially-ordered Patterns}

\author[K. Yap]{Kai Ting Keshia Yap}
\address{Department of Mathematics \& Statistics, Queens University,
48 University Ave. Jeffery Hall Kingston, ON Canada K7L 3N6}
\email{yap.keshia@gmail.com}

\author[D. Wehlau]{David Wehlau$^*$}
\address{Department of Mathematics \& Computer Science, Royal Military College of Canada,
P.O.Box 17000, Station Forces, Kingston, Ontario, Canada K7K 7B4}
\email{David.Wehlau@rmcc-cmrc.ca}
\thanks{$^*$ Partially supported by NSERC}

\author[I. Zaguia]{Imed Zaguia$^{*1}$}
\address{Department of Mathematics \& Computer Science, Royal Military College of Canada,
P.O.Box 17000, Station Forces, Kingston, Ontario, Canada K7K 7B4}
\email{Imed.Zaguia@rmcc-cmrc.ca}
\thanks{$^1$ Corresponding author.}

\date{\today}

\keywords{bijection; pattern avoidance; permutation; POP avoidance; simple permutation}
\subjclass[2010]{05A05, 05A15}

\begin{abstract}A permutation $\pi$ contains a pattern $\sigma$ if and only if there is a subsequence in $\pi$ with its letters are in the same relative order as those in $\sigma$. Partially ordered patterns (POPs) provide a convenient way to denote patterns in which the relative order of some of the letters does not matter. This paper elucidates connections between the avoidance sets of a few POPs with other combinatorial objects, directly answering five open questions posed by Gao and Kitaev \cite{gao-kitaev-2019}. This was done by thoroughly analysing the avoidance sets and developing recursive algorithms to derive these sets and their corresponding combinatorial objects in parallel, which yielded a natural bijection. We also analysed an avoidance set whose simple permutations are enumerated by the Fibonacci numbers and derived an algorithm to obtain them recursively.
\end{abstract}

\maketitle

\section{Introduction}\label{chap:intro}
This paper elucidates connections between the avoidance sets of a few Partially Ordered Patterns (POPs) with other combinatorial objects, directly answering five open questions posed by Gao and Kitaev \cite{gao-kitaev-2019}. Results in this article appeared in the first author's MSC dissertation.

We write $[n]$ to denote the set of integers $\{1,2,\dots,n\}$ for $n\geq 1$.
A \textit{permutation} is a bijection from $[n]$ to itself for some $n\geq 1$.
We call such a permutation an \textit{$n$-permutation}
and typically denote it by $\pi=\pi_1 \,\pi_2 \,\dots \,\pi_n$, where $\pi_i=\pi(i)$.
We say that its \textit{length} or \textit{size} (denoted $\abs{\pi}$) is $n$.
We write $S_n$ to denote the set of all $n$-permutations.
We denote the size of any set $S$ by $\# S$ or $\abs{S}$.

\subsection{Background}
A permutation $\pi$ \textit{contains} a \textit{pattern} $\sigma$
if and only if there is a \textit{subsequence} in $\pi$ (of the same length as $\sigma$)
with its letters are in the same relative order as those in $\sigma$.
For instance, the pattern $312$ occurs in 42531 (as the subsequence 423), but not in 132465.
The permutations that avoid a pattern or a set of patterns make up an \textit{avoidance set}.
Avoidance sets have been studied extensively
and research in this area has important applications to numerous fields.
Examples include
sorting devices in theoretical computer science,
Schubert varieties and Kazhdan-Lusztig polynomials,
statistical mechanics,
the tandem duplication-random loss model in computational biology
and bijective combinatorics
(see \cite{kitaev-textbook} and references therein).\\

A \textit{partially ordered pattern} (abbreviated \textit{POP}) is a
\textit{partially ordered set} (\textit{poset}) that generalizes the notion of a pattern
when we are not concerned with the relative order of some of its letters,
and therefore may represent multiple patterns.
Specifically, a POP is a poset with $n$ elements labelled $1,\,2,\,\dots,\, n$,
for some $n\geq 1$.
For any pattern that the POP represents,
the partial order of the elements stipulates the \textit{relative order} of letters in the pattern,
where the labels of the elements indicate the \textit{positional order} of these letters.
For example, the POP $p=$ \begin{tikzpicture}
	\begin{pgfonlayer}{nodelayer}
		\node [style=node] (0) at (0, 1.2) {};
		\node [style=node] (1) at (-1, 0) {};
		\node [style=node] (2) at (0, 0) {};
		\node [style=node] (3) at (1, 0) {};
		\node [style=none] (5) at (0, 2) {1};
		\node [style=none] (6) at (-1, -0.8) {2};
		\node [style=none] (7) at (0, -0.8) {3};
		\node [style=none] (8) at (1, -0.8) {4};
	\end{pgfonlayer}
	\begin{pgfonlayer}{edgelayer}
		\draw (0.center) to (2.center);
	\end{pgfonlayer}
\end{tikzpicture} represents all the patterns of length 4
whose first element is larger than the third element.
That is, $p$ represents the twelve patterns \[2314,\,2413,\,3124,\,3421,\,3214,\,3412,\,4213,\,4312,\,4123,\,4321,\,4132\,\text{and }\,4231.\]
A POP may represent a single pattern.
For example, the pattern 3241 represented as a POP is the chain
of four elements labelled 1, 2, 3 and 4 with the order $4<2<1<3$.
Note that 3241 is the permutation inverse of 4213, and this is not a coincidence.\\
A permutation \textit{contains} a POP if and only if it contains at least one of the patterns represented by that POP.
Otherwise, it \textit{avoids} the POP.
For example, the permutation 3472615 contains 21 occurrences of the POP $p$ (defined above) whereas 132456 avoids $p$.

\subsection{Motivation and structure}
Enumerating the permutations of different lengths in the avoidance set of a pattern or set of patterns
and finding one-to-one correspondences to well-known combinatorial objects
is a topic of great interest.
Several classical combinatorial objects may be related to a single avoidance set,
and finding these connections would allow us to
understand seemingly disparate objects under a common framework \cite{kitaev-textbook}.
With the aid of a computer software, Gao and Kitaev \cite{gao-kitaev-2019}
conducted a systematic search of connections between sequences in The Online Encyclopedia of Integer Sequences (OEIS) \cite{oeis}
and the enumeration of permutations avoiding POPs with 4 or 5 elements.
They observed connections to 38 sequences in OEIS and
listed 15 combinatorial objects with which potentially interesting bijections might occur
with the avoidance sets of certain POPs
(see Tables 6 and 7 of their paper).\\

The goal of this paper is to find as many bijections between the pairs of objects in ways that are meaningful.
With the help of an interactive software \href{http://www.cs.otago.ac.nz/PermLab/}{PermLab} \cite{permlab},
we successfully construct nontrivial bijections for five of these pairs
and find generalizations whenever possible.
We list the objects in Table \ref{table:summary}
and discuss the bijections in Sections \ref{chap:av1}, \ref{chap:Q}, \ref{chap:levels} and \ref{chap:av10}.
One bijection (discussed in Section \ref{sect:juggling}) emerges directly from the original proof of the enumeration of ground-state juggling sequences by Chung and Graham \cite{chung-graham}.
For each of the remaining four bijections,
we first realize that both sets in the corresponding pair
could be partitioned into subsets of corresponding sizes.
This allows us to construct similar recursive algorithms that can build the sets in parallel,
which in turn yield (one or many) bijections that could be constructed directly and explicitly.
Thus, we end up with
a thorough understanding of the permutations that avoid each POP and of the corresponding combinatorial objects.\\

During our analysis, we discovered a set of patterns that are avoided by infinitely many \textit{simple permutations} (to be defined in Section \ref{chap:prelim}),
which are, in fact, enumerated by a translation of the well-known \textit{Fibonacci sequence}.
We construct an algorithm that allows one to obtain this set of permutations recursively
and prove this in Section \ref{chap:fib}.
Section \ref{chap:prelim} defines all the relevant terms and concepts in detail
and Section \ref{chap:conclusion} summarises our research and lists possible avenues of further research.


\begin{table}
    \centering
    \begin{tabular}{|c|c|c|c|}
        \hline
            \thead{POP}  & \thead{OEIS sequence \\(beginning with $n=1$)} & \thead{Equinumerous structures} & \thead{Location} \\ \hline
            \makecell{\begin{tikzpicture}
	\begin{pgfonlayer}{nodelayer}
		\node [style=node] (0) at (0, 1.5) {};
		\node [style=node] (1) at (-1.25, 0) {};
		\node [style=node] (2) at (1.25, 0) {};
		\node [style=none] (4) at (0, 2.5) {1};
		\node [style=none] (5) at (-1.25, -1) {2};
		\node [style=none] (6) at (0, -1) {3};
		\node [style=none] (7) at (1.25, -1) {4};
		\node [style=node] (8) at (0, 0) {};
	\end{pgfonlayer}
	\begin{pgfonlayer}{edgelayer}
		\draw (0.center) to (1.center);
		\draw (0.center) to (2.center);
	\end{pgfonlayer}
\end{tikzpicture}}  & \makecell{\href{https://oeis.org/A111281}{A111281}\\1, 2, 6, 16, 40, 100, 252,\\636, 1604, 4044,\\10196,25708, ...} & \makecell{permutations \\avoiding the patterns \\ 2413, 2431, 4213, 3412, \\ 3421, 4231, 4321, 4312}   & Section \ref{chap:av1} \\
            \makecell{\begin{tikzpicture}
	\begin{pgfonlayer}{nodelayer}
		\node [style=node] (0) at (0, 1.7) {};
		\node [style=node] (1) at (-1.525, 0) {};
		\node [style=node] (2) at (-0.525, 0) {};
		\node [style=node] (3) at (0.575, 0) {};
		\node [style=node] (4) at (1.55, 0) {};
		\node [style=none] (5) at (0, 2.675) {1};
		\node [style=none] (6) at (-1.5, -1) {2};
		\node [style=none] (7) at (1.55, -1) {5};
		\node [style=none] (8) at (0.6, -1) {4};
		\node [style=none] (9) at (-0.5, -1) {3};
	\end{pgfonlayer}
	\begin{pgfonlayer}{edgelayer}
		\draw (0.center) to (1.center);
		\draw (0.center) to (2.center);
		\draw (0.center) to (3.center);
		\draw (0.center) to (4.center);
	\end{pgfonlayer}
\end{tikzpicture}} & \makecell{\href{https://oeis.org/A084509}{A084509}\\1, 2, 6, 24, 96, 384, 1536,\\6144, 24576, 98304,\\393216, 1572864, ...} & \makecell{number of ground-state \\ 3-ball juggling sequences \\ of period $n$} & Section \ref{sect:juggling} \\
            \makecell{\begin{tikzpicture}
	\begin{pgfonlayer}{nodelayer}
		\node [style=node] (0) at (-0.925, 0) {};
		\node [style=node] (1) at (0, 0) {};
		\node [style=node] (2) at (0, 1.525) {};
		\node [style=node] (3) at (0.975, 0) {};
		\node [style=none] (4) at (0.975, -1) {4};
		\node [style=none] (5) at (-1.025, -1) {2};
		\node [style=none] (6) at (-0.025, -1) {3};
		\node [style=none] (7) at (-0.025, 2.5) {1};
	\end{pgfonlayer}
	\begin{pgfonlayer}{edgelayer}
		\draw (2.center) to (0.center);
		\draw (2.center) to (1.center);
		\draw (2.center) to (3.center);
	\end{pgfonlayer}
\end{tikzpicture}}  & \makecell{\href{https://oeis.org/A025192}{A025192}\\1, 2, 6, 18, 54, 162, 486,\\1458, 4374, 13122,\\39366, 118098, ...} & \makecell{2-ary shrub forests\\of $n$ heaps avoiding\\the patterns 231, 312, 321}  & Section \ref{sect:shrub} \\
            \makecell{\begin{tikzpicture}
	\begin{pgfonlayer}{nodelayer}
		\node [style=node] (0) at (0, 1.6) {};
		\node [style=node] (1) at (0, 0) {};
		\node [style=node] (2) at (-1, 0) {};
		\node [style=node] (3) at (1, 0) {};
		\node [style=none] (4) at (0, 2.5) {1};
		\node [style=none] (5) at (0, -1) {3};
		\node [style=none] (6) at (-1, -1) {2};
		\node [style=none] (7) at (1, -1) {4};
	\end{pgfonlayer}
	\begin{pgfonlayer}{edgelayer}
		\draw (0.center) to (1.center);
	\end{pgfonlayer}
\end{tikzpicture}}  & \makecell{\href{https://oeis.org/A045925}{A045925}\\1, 2, 6, 12, 25,48, 91,\\168, 306, 550, 979,\\1728, 3029...} & \makecell{levels in all \\compositions of $n+1$ \\with only ones and twos} & Section \ref{sect:av9-levels} \\
            \makecell{\begin{tikzpicture}
	\begin{pgfonlayer}{nodelayer}
		\node [style=node] (0) at (0, 1.35) {};
		\node [style=node] (1) at (0, 0.025) {};
		\node [style=none] (2) at (1.075, -1.025) {2};
		\node [style=none] (3) at (2.075, -1.025) {3};
		\node [style=none] (4) at (0, -1) {4};
		\node [style=none] (5) at (0, 2.325) {1};
		\node [style=node] (6) at (1.05, -0.05) {};
		\node [style=node] (7) at (2.05, 0) {};
	\end{pgfonlayer}
	\begin{pgfonlayer}{edgelayer}
		\draw (0.center) to (1.center);
	\end{pgfonlayer}
\end{tikzpicture}} & \makecell{\href{https://oeis.org/A214663}{A214663} and \href{https://oeis.org/A232164}{A232164}\\1, 2, 6, 12, 25, 57, 124,\\268, 588, 1285, 2801,\\6118, 13362, ...} & \makecell{number of $n$-permutations \\ for which the partial sums \\of signed displacements \\ do not exceed 2}  & Section \ref{chap:av10} \\
        \hline
    \end{tabular}
    \caption{List of POPs studied}
    \label{table:summary}
\end{table}

\break 

\subsection{Preliminaries}\label{chap:prelim}
For an $n$-permutation $\pi$,
    we say that $\pi_{i_1}\, \pi_{i_2}\,\cdots\,\pi_{i_k}$
    is a \textit{subsequence} of $\pi$ if and only if
    $1\leq i_1 < i_2 < \cdots i_k \leq n$ and $k\in [n]$.
    For an $n$-permutation $\pi$ and any $1\leq i,\, j\leq n$,
    the contiguous substring $\pi_i\pi_{i+1}\,\cdots\,\pi_j$
    is called a \textit{factor} of $\pi$.
    We denote $\pi_i\pi_{i+1}\,\cdots\,\pi_j$ simply as $\pi_{[i,j]}$.
    Note that if $i=j$, then $\pi_{[i,j]}=\pi_i$ has length 1
    and we call it a \textit{point}, \textit{term} or an \textit{element}. If $i>j$ then $\pi_{[i,j]}$ has length 0,
    and we say that it is \textit{empty}.
    Let $\alpha=\pi_{[i_1,j_1]}$ and $\beta=\pi_{[i_2,j_2]}$ be non-empty factors of $\pi$.
    We write $\alpha<\beta$ if and only if
    $\pi_{\ell_1}<\pi_{\ell_2}$ for all $\ell_1\in [i_1,j_1]$ and $\ell_2\in [i_2,j_2]$.\\

    Note that we may extend the definition of factors of permutations to factors of factors.
    If $\pi$ is a factor of size $n$ of a larger $m$-permutation $\zeta$,
    say $\pi=\zeta_{[i,j]}$ for some $1\leq i\leq j\leq m$,
    then we use $\pi_k$ to denote $\zeta_{i+k-1}$ for any $k\in [n]$.
    Then $\pi_{[k,\ell]} = \zeta_{[i+k-1,\, i_\ell-1]}$ for any $1\leq k\leq \ell\leq n$.\\


    We say that a factor $\sigma=\pi_{[i,j]}$ (for some $1\leq i\leq j\leq n$) of an $n$-permutation $\pi$
    \textit{contains the number} $x$, denoted as $x\in \sigma$,
    if and only if $\pi_{\ell}=x$ for some $\ell\in [i,j]$.
    Otherwise, we say that $\sigma$ does not contain the number $x$
    and write $x\not\in \sigma$.\\

    For an $n$-permutation $\pi$, we say that a factor $\pi_{[i,j]}$
    is an \textit{interval} if and only if it contains exactly the numbers in a contiguous interval of $[n]$.
    That is, if and only if
    $\{\pi_\ell \mid \ell\in [i,j]\} = [s,t]$ for some $s,t\in [n]$.
    For example, the factor 2413 is an interval
    while the factor 241 is not an interval.
    An interval of an $n$-permutation is \textit{trivial}
    if and only if its length is 0, 1 or $n$.\\
%

    Let $\pi_{i_1}\pi_{i_2}\,\cdots\,\pi_{i_k}$ be a subsequence of an $n$-permutation $\pi$.
    The \textit{reduced} subsequence
    $\text{red}(\pi_{i_1}\pi_{i_2}\,\cdots\,\pi_{i_k})$
    is defined as the $k$-permutation that is order-isomorphic to the subsequence.
    That is, $\text{red}(\pi_{i_1}\, \pi_{i_2}\,\cdots\,\pi_{i_k})=\sigma$ is the $k$-permutation
    where $\sigma_{s}<\sigma_{t}$ if and only if $\pi_{i_s}<\pi_{i_t}$
    We say that $\sigma$ is the \textit{reduction} of the subsequence $\pi_{i_1}\pi_{i_2}\,\cdots\,\pi_{i_k}$.

\begin{example}
    Let $\pi=1435726$.
    Then the following statements are true:
    \begin{enumerate}[(a)]
        \item $1576$ is a subsequence of $\pi$
        \item $\text{red}(1576)=1243$
        \item $\pi_{[3,5]}=357$ and $\pi_{[6,7]}=26$ are factors of $\pi$
        \item $\pi_{[2,4]}=435$ is an interval of $\pi$
    \end{enumerate}
\end{example}

\subsection{Simple permutations}\label{sect:intro-simple}

\begin{definition}
    An $n$-permutation is \textit{simple} if and only if
    all its intervals are trivial. That is, if and only if
    its intervals are all of length 0, 1 or $n$.
\end{definition}

Simple permutations were first considered in \cite{ivo}.

\begin{definition}
    Let $\sigma$ be a $k$-permutation,
        and for $\ell\in [k]$,
        let $\alpha^{(\ell)}$ be a permutation of length $i_\ell$.
    We define the \textit{inflation} of $\sigma$ by $\alpha^{(1)}, \alpha^{(2)},\dots, \alpha^{(k)}$
    as the permutation \[\pi=\sigma[\alpha^{(1)}, \alpha^{(2)},\dots, \alpha^{(k)}]\]
    of length $n := i_1+i_2+\cdots + i_k$
    where given $s_0:=0$, $\ell\in [k]$, $s_\ell := i_1+i_2+\cdots + i_\ell$,
    the following hold:
    \begin{itemize}
        \item the factors $\pi_{[s_{\ell-1}+1,\, s_{\ell}]}$ are intervals
        \item $\text{red}(\pi_{[s_{\ell-1}+1,\, s_{\ell}]}) = \alpha^{(\ell)}$
        \item $\pi_{[s_{t-1}+1,\, s_t]}<\pi_{[s_{u-1}+1,\, s_u]}$ if and only if $\sigma_t < \sigma_u$.
    \end{itemize}
    We call $\sigma$ a \textit{quotient} of $\pi$.
\end{definition}

\begin{example}
    The permutation 526314 is simple while the 4215763 is not.
    The following statements are true:
    \begin{enumerate}[(a)]
        \item $4215763=3142[1,21,132,1]$. Note 4, 21, 576 and 3 are intervals of 4215763.
        \item 3142 is a quotient of 4215763
        \item 4215763 is an inflation of 3142
    \end{enumerate}
\end{example}

\begin{prop}[Albert and Atkinson \cite{albert-atkinson}]
    Every permutation may be written as the inflation of a unique simple permutation.
    Moreover, if $\pi$ can be written as $\sigma[\alpha^{(1)},\, \alpha^{(2)},\, \dots,\, \alpha^{(m)}]$
    where $\sigma$ is simple and $m\geq 4$,
    then the $\alpha^{(i)}$s are unique.
\end{prop}

\subsection{Separable permutations}

\begin{definition}
    Suppose $\pi$ and $\sigma$ are permutations of length $n$ and $m$ respectively.
    We define the \textit{direct sum} (or simply, \textit{sum}), using the operator $\oplus$, and the skew sum, using the operator $\ominus$,
    of $\pi$ and $\sigma$ as the permutations of length $m+n$ as follows:
    \[\pi \oplus \sigma = 12[\pi,\sigma]
    \quad \text{and}\quad
    \pi \ominus \sigma = 21[\pi,\sigma].\]
    \end{definition}

\begin{definition}
    If a permutation is an inflation of 12 or 21,
    we call it
    \textit{sum decomposable} and \textit{skew sum decomposable} respectively.
    If a permutation is not sum decomposable we say it is \textit{sum indecomposable},
    and if it is not skew sum decomposable we say it is \textit{skew-sum indecomposable}.
\end{definition}

\begin{definition} 
    A permutation is \textit{separable} if it can be obtained
    by repeatedly applying the $\oplus$ and $\ominus$ operations on the permutation 1.
\end{definition}

\begin{example}
    The permutation 587694231 is separable, since
    \begin{align*}
        587694231
        &=14325\ominus 4231\\
        &=(1\oplus 3214)\ominus (1\ominus 231)\\
        &=(1\oplus (321\oplus 1)\ominus (1\ominus (12\ominus 1)))\\
        &=(1\oplus ((1\ominus 21)\oplus 1)\ominus (1\ominus (12\ominus 1)))\\
        &=(1\oplus ((1\ominus (1\ominus 1))\oplus 1)\ominus (1\ominus ((1\oplus 1)\ominus 1))).
    \end{align*}
\end{example}

\begin{example}
    All permutations of length 3 are separable.
    The only permutations of length 4 that are not separable are 2413 and 3142.
\end{example}


\begin{theorem}[\textit{folklore}]\label{thm:separable}
A  permutation is separable if and only if it avoids 2413 and 3412.
\end{theorem}

\begin{prop}[Albert and Atkinson \cite{albert-atkinson}]\label{thm:12,21}
    If $\pi$ is an inflation of 12,
    say $\pi = 12[\alpha,\, \beta]$,
    then $\alpha$ and $\beta$ are unique if $\alpha$ is sum indecomposable.
    The same holds with 12 replaced by 21 and ``sum'' replaced by ``skew sum''.
\end{prop}

\begin{corollary}\label{cor:separable}
    All simple permutations of length at least 4 must contain the patterns 132, 213, 231 and 312.
\end{corollary}
\begin{proof}
    It is clear that simple permutations are not separable,
    so by Theorem \ref*{thm:separable}, they must contain at least one of 2413 or 3412,
    both of which contain the four patterns of length three.
\end{proof}

\subsection{Pattern/POP containment and avoidance}\label{sect:intro-avoidance}

\begin{definition} 
    A \textit{pattern} is a permutation of length at least 2.
    We say that a permutation $\pi$ \textit{contains} a pattern $p$ if and only if
    there exists some subsequence $\pi_{i_1}\, \pi_{i_2}\,\cdots\,\pi_{i_k}$ of $\pi$
    where $\text{red}(\pi_{i_1}\, \pi_{i_2}\,\cdots\,\pi_{i_k})=p$.
    That is, $p_j<p_\ell$ if and only if $\pi_{i_j}<\pi_{i_\ell}$
    for all $j,\, \ell\in [k]$.
    Otherwise, we say that $\pi$ \textit{avoids} $p$.
    If $P$ is a set of patterns, we say that $\pi$ \textit{contains} $P$
    if $\pi$ contains any pattern in $P$. Otherwise we say that $\pi$ \textit{avoids} $P$.
\end{definition}



A partially ordered pattern generalizes the notion of a pattern whereby
the order between certain elements do not have to be considered.
We are left with a partial order on the elements,
which we can represent using a labelled partially ordered set.
Recall that a \textit{partial order} is a binary relation $\leq $ over a set $P$
that is \textit{reflexive}, \textit{antisymmetric} and \textit{transitive}.
That is, for all $a,b,c\in P$, the following hold:
\begin{enumerate}[1.]
    \item $a\leq a$ (\emph{reflexivity});
    \item If $a\leq b$ and $b\leq a$ then $a=b$ (\emph{antisymmetry});
    \item If $a\leq b$ and $b\leq c$ then $a<c$ (\emph{transitivity}).
\end{enumerate}
A set $P$ with a partial order $\leq$ is called a \textit{partially ordered set (poset)}, denoted $(P,\leq)$.
We may write $b\geq a$ as an equivalent statement to $a\leq b$ for any $a,\, b\in P$.
We write $a<b$ to mean that $a\leq b$ and that $a$ and $b$ are distinct.

\begin{definition} 
    A \textit{partially ordered pattern (POP)} $p$ of \textit{size} $k$ is a poset
    with $k$ elements labelled $1,\,2,\,\dots,\, k$.
    A POP can be expressed in one-line notation by indicating its size and
    the minimal set of relations that defines the respective poset.
\end{definition}

\begin{definition}
    An $n$-permutation $\pi$ \textit{contains} such a POP $p$ if and only if
    $\pi$ has a subsequence $\pi_{i_1}\pi_{i_2}\cdots \pi_{i_k}$
    such that $\pi_{i_j} < \pi_{i_m}$ if $j<m$ in the poset $P$.
    Otherwise, we say that $\pi$ \textit{avoids} $p$.
\end{definition}


\begin{example}
    The pattern 3241 represented as a POP is the 4-element chain with its elements labelled 1, 2, 3 and 4, where $4<2<1<3$.
    Note that the permutation 4213 is the inverse of 3241.
\end{example}

Recall that a poset can be represented visually as a \textit{Hasse diagram}.
A Hasse diagram of a finite poset is a visual representation of the elements and relations in the poset,
where only the \textit{covering relations} are shown.
Recall that a covering relation in a poset $(P,\leq )$ is a binary relation
$i\prec j$ for some $i$ and $j$ in $P$
where $i<j$ and there does not exist any $k\in P$ such that both $i<k$ and $k<j$ hold.
A Hasse diagram uniquely determines the partial order.\\

A Hasse diagram of a poset with $n$ elements can also be understood as a simple directed graph $(V,E)$ with an implicit upward orientation
where $V$ is a set of $n$ vertices and $E$ is a set of ordered pairs of distinct elements in $V$,
i.e. $E\subseteq \{(i,j)\mid i,j\in V,\, i\neq j\}$, that satisfies the following three conditions:
\begin{enumerate}[1.]
    \item if $(i,j)$ is in $E$ then $(j,i)$ is not in $E$ (antisymmetry),
    \item if $(i,j)$ and $(j,k)$ are in $E$ then $(i,k)$ is not in $E$ (transitive reduction),
    \item if $(i,j)$ is in $E$ then $i\leq j$ is a relation in the poset.
\end{enumerate}
A POP is a labelled poset, and can therefore be represented visually as a labelled Hasse diagram.
That is, as a graph $(V,E)$ defined as above where the vertices in $V$ are labelled $1,\,2,\,\dots,\, k$.

\begin{example}\label{eg:POP-example}
    The POP of size 4 where $1>3$ is illustrated in Figure \ref{fig:POP-example}.
    It represents the patterns 2314, 2413, 3124, 3421, 3214, 3412, 4213, 4312, 4123, 4321, 4132, 4231.
    The permutation 3472615 contains 21 occurrences of the POP whereas 132456 avoids it.
\end{example}

\begin{figure}[!htbp]
    \centering
    \begin{tikzpicture}
	\begin{pgfonlayer}{nodelayer}
		\node [style=node] (0) at (0, 1.2) {};
		\node [style=node] (1) at (-1, 0) {};
		\node [style=node] (2) at (0, 0) {};
		\node [style=node] (3) at (1, 0) {};
		\node [style=none] (5) at (0, 2) {1};
		\node [style=none] (6) at (-1, -0.8) {2};
		\node [style=none] (7) at (0, -0.8) {3};
		\node [style=none] (8) at (1, -0.8) {4};
	\end{pgfonlayer}
	\begin{pgfonlayer}{edgelayer}
		\draw (0.center) to (2.center);
	\end{pgfonlayer}
\end{tikzpicture}
    \caption{The Hasse diagram representation of the POP in Example \ref{eg:POP-example}}
    \label{fig:POP-example}
\end{figure}
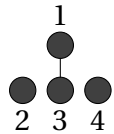

\begin{definition}
    Let $P$ be a pattern, a set of patterns, or a POP.
    We denote $Av(P)$ as the set of permutations that avoid $P$ (called the \textit{avoidance set of $P$}),
    and $Av_n(P)$ as the set of $n$-permutations that avoid $P$.
    That is, $Av_n(P):=Av(P)\cap S_n$.
\end{definition}

\begin{definition}
    Let $P_1$ and $P_2$ each be a set of patterns of a POP. We say that $P_1$ and $P_2$ are
    \textit{Wilf-equivalent} if and only if $\abs{Av_n(P_1)}=\abs{Av_n(P_2)}$ for all $n\geq 1$.
\end{definition}

It is not hard to check that containment is a partial order on any set of permutations.
In the literature, sets of permutations which are \textit{closed downward} under this order are called
\textit{permutation classes}, or sometimes just \textit{classes}.
That is, $\mathcal{C}$ is a \textit{permutation class} if and only if
for any $\pi\in \mathcal{C}$ and any $\sigma$ contained in $\pi$,
we have $\sigma\in \mathcal{C}$.
If a permutation $\pi$ avoids a pattern $p$,
then every reduced subsequence of $\pi$ avoids $p$.
In other words, every pattern contained in $\pi$ avoids $p$.
Therefore $Av(p)$ and $Av_n(p)$ are permutation classes.
The same is true if $p$ is a set of patterns or a POP.\\


Observe that if $k_P$ is the length of the shortest pattern in a set of patterns $P$,
then all permutations of length less than $k_P$ avoid $P$.
This means that $\abs{Av_n(P)}=n!$ for all $n<k_P$.
Therefore it suffices to enumerate $Av_n(P)$ for $n\geq k_P$
for every POP or set of patterns $P$ discussed in subsequent sections.

\subsection{Matrix representations of permutations}


\begin{definition}
    For an $n$-permutation $\pi$,
    its \textit{permutation matrix}
    is a binary $n\times n$ matrix, denoted $M(\pi)$ where
    \[M(\pi)_{n-i+1,j}=1 \iff \pi_j=i.\]
    Moreover, its \textit{pattern matrix}
    is an $n\times n$ matrix, denoted $M'(\pi)$, where
    \[M'(\pi)_{n-i+1,j}=\begin{cases}
        i&\text{ if }\pi_j=i,\\
        0&\text{ otherwise}.
    \end{cases}\]
    We may refer to the non-zero entries in a permutation matrix or pattern matrix as \textit{points}.
    Sometimes, we may omit displaying the 0s and the traditional brackets
    if no confusion would arise.

\end{definition}
\begin{example}
    Let $\pi=312$. Its permutation matrix is
    \[M(\pi)\quad
    =\quad\begin{pmatrix}
        1&0&0\\
        0&0&1\\
        0&1&0
    \end{pmatrix}\quad
    =\quad\begin{matrix}
        1& & \\
         & &1\\
         &1&
    \end{matrix}\]
    and its pattern matrix is
    \[M'(\pi)\quad
    =\quad\begin{pmatrix}
        3&0&0\\
        0&0&2\\
        0&1&0
    \end{pmatrix}\quad
    =\quad\begin{matrix}
        3& & \\
         & &2\\
         &1&
    \end{matrix}.\]
\end{example}

\begin{definition}
    The \textit{weight} of a matrix is the number of non-zero entries it contains. We denote the weight of a matrix $A$ by $\abs{A}$.
\end{definition}

\begin{example}
    The weight of the permutation matrix of $\pi$
    is equal to the length of $\pi$.
    The weight of any column or row of a permutation matrix is 1.
\end{example}

\subsection{Lattice matrices}

\begin{definition}
    We call a matrix (or submatrix) \textit{void} if it has no rows or no columns.
    A matrix (or submatrix) is \textit{trivial} if its weight is 0, and \textit{nontrivial} otherwise.
    That is, a trivial matrix is either void or is a zero matrix.
    All void matrices are trivial.
\end{definition}

If we know that a permutation contains a certain pattern,
it might be helpful to represent its permutation as a block matrix
in order to better understand the permutation.
We will show an example before stating formal definitions:\\

Suppose we know that the $n$-permutation $\pi$ contains the pattern $p:=312$.
That is, there exist $i$, $j$ and $k$ where $1\leq i<j<k\leq n$ and $\pi_i\, \pi_j\, \pi_k$ reduces to 312.
The columns $i$, $j$ and $k$ partition the rest of the permutation matrix $M(\pi)$ into 4 (possibly trivial) blocks of columns,
and the rows $\pi_i,\, \pi_j$ and $\pi_k$ partition the rest of the permutation matrix $M(\pi)$ into 4 (possibly trivial) blocks of rows.
This gives rise to another representation of $M(\pi)$ as a $7\times 7$ block matrix.
In this representation we can find the $3\times 3$ matrix $M(p)$
interwoven with a $4\times 4$ block matrix $(\alpha_{ij})_{i,j\in [4]}$
as depicted in Figure \ref{fig:lattice},
with the following alterations:
\begin{itemize}
    \item the ones in columns $i,$ $j$ and $k$ are replaced by $\pi_i$, $\pi_j$ and $\pi_k$ respectively, \\
    in other words, we replace the submatrix corresponding $M(p)$ with the pattern matrix $M'(p)$
    \item the zeroes in columns $i,$ $j$ and $k$ that are also in row $\pi_i$, $\pi_j$ or $\pi_k$
    are replaced by plus signs,
    \item the trivial blocks in columns $i,$ $j$ and $k$ are replaced by vertical bars,
    \item the trivial blocks in rows $\pi_i$, $\pi_j$ and $\pi_k$ are replaced by horizontal bars, and finally,
    \item the conventional matrix brackets are omitted.
\end{itemize}

\noindent Note that we may also alter the ones, zeroes and $\alpha_{ij}$ blocks (where $i,j\in[n+1]$)
differently based on which properties of the permutation we are trying to highlight.
This figure is reminiscent of the lattice structure of gridded window panes,
so we call it the \textit{$p$-lattice matrix of $\pi$}, or simply a \textit{lattice matrix},
and denote it by $L_p(\pi)$.\\

\begin{figure}
    \begin{center}
        \begin{tabular}{ccccccc}
            $\alpha_{11}$ &   \Big |  & $\alpha_{12}$ &   \Big |  & $\alpha_{13}$ &   \Big |  & $\alpha_{14}$\\
            --------      &     3     &    --------   &  $\Plus$  &    --------   &  $\Plus$  & --------\\
            $\alpha_{21}$ &   \Big |  & $\alpha_{22}$ &   \Big |  & $\alpha_{23}$ &   \Big |  & $\alpha_{24}$\\
            --------      &  $\Plus$  &    --------   &  $\Plus$  &    --------   &      2    & --------\\
            $\alpha_{31}$ &   \Big |  & $\alpha_{32}$ &   \Big |  & $\alpha_{33}$ &   \Big |  & $\alpha_{34}$\\
            --------      &  $\Plus$  &    --------   &      1    &    --------   &  $\Plus$  & --------\\
            $\alpha_{41}$ &   \Big |  & $\alpha_{42}$ &   \Big |  & $\alpha_{43}$ &   \Big |  & $\alpha_{44}$\\
        \end{tabular}
        \caption{The lattice matrix $L_{312}(\pi)$}
        \label{fig:lattice}
    \end{center}
\end{figure}

\noindent In general, we may use the following definition:

\begin{definition}
Consider a permutation $\pi$ on $n$ letters and choose
$m$ indices $1\leq i_1<i_2<\cdots <i_m\leq n$.
Write $I=(i_1,\, i_2,\, \dots,\, i_m)$ and let $p:=\text{red}(\pi_{i_1}\pi_{i_2}\, \cdots\, \pi_{i_m})$.
We proceed to define the lattice matrix $L_p(\pi)$.
Put $i_0=0$ and $i_{m+1}=n+1$.
We use the values $i_1,\, i_2,\, \dots,\, i_m$ to partition the column indices into subintervals
and the values $\pi_{i_1},\,\pi_{i_2},\,\dots,\,\pi_{i_m}$
to partition the row indices into subintervals.\\

A \textit{block} in $L_p(\pi)$ is a (possibly trivial) continguous block submatrix of the permutation matrix $M(\pi)$
with its column and row indices each given by a relevant subinterval defined above.
To make this more explicit, we note that $\pi(i_{p\inv(1)})< \pi(i_{p\inv(2)}) < \cdots < \pi(i_{p\inv(m)})$.
Write $j_k:= \pi(i_{p\inv(k)})$.
Put $j_0= 0$ and $j_{m+1}=n+1$.
Then $M_{S\times T}$ where $S=[i_{s-1}+1,\, i_s -1]$ and $T=[j_{t-1}+1,\, j_t -1]$
is a \textit{block}
for all $s$ and $t$ in $[m+ 1]$.
We label the block $M_{S\times T}$ as $\alpha_{s,t}$.
Note that $\alpha_{s,t}$ is void if either $i_s=i_{s-1}$ or $j_t=j_{t-1}+1$.
We may write $\alpha_{s,t}$ as $\alpha_{st}$ if no confusion would arise.
Note that $M_{i_k, \pi(i_k)}$, $M_{i_k\times T}$ and $M_{S\times j_{k}}$ may also be referred to as blocks
for any $k\in [m]$ and subintervals $S$ and $T$ defined as above.\\

When depicting $L_p(\pi)$, we replace the zero entries in
column $i_k$ by a vertical line and the zero entries in row $j_k$ by a horizontal line for every $k\in [m+1]$.
We also omit the traditional parentheses around the entire matrix $L_p(\pi)$.
\end{definition}

\break

\begin{prop}
    It is easy to deduce the following properties of the lattice matrix $L_p(\pi)$ defined above:
    \begin{enumerate}[1.]
        \item $(\alpha_{s,t})_{a,b} = M_{j_{s-1}+a,\, i_{t-1}+b}$
        for $a \in [j_s - j_{s-1}-1]$ and $b\in [i_{t} - i_{t-1}-1]$
        \item Each horizontal (respectively, vertical) bar is either void, or is a single row (respectively, column) of zeroes.
        \item All block matrices in the same row (respectively, column) of the lattice matrix have the same number of rows (respectively, columns).
        \item Any square block submatrix of the lattice matrix has the same total number of rows as columns.
    \end{enumerate}
\end{prop}

We would like to describe the relationships between points in different blocks in a lattice matrix precisely.
The following definitions provide an intuitive way to do so.

\begin{definition}
    Let $p$ and $\pi$ be permutations of length $m$ and $n$ respectively where $m\leq n$ and $\pi$ contains $p$.
    Let $L_p(\pi)$ be the $p$-lattice matrix of $\pi$ with its blocks denoted by $\alpha_{ij}$ for $i,\, j\in [m+1]$.

    \begin{enumerate}[(a)]
        \item For the points in $L_p(\pi)$ that correspond to the pattern $p$, we define a point being \textit{adjacent} to a block in a natural way.
        For example, in Figure \ref{fig:lattice}, the point labelled 1 is adjacent to the blocks $\alpha_{32},\,\alpha_{33},\,\alpha_{42}$ and $\alpha_{43}$
        while the point labelled 2 is adjacent to the blocks $\alpha_{23},\,\alpha_{24},\,\alpha_{33}$ and $\alpha_{34}$.
        We note that every point is adjacent to exactly four blocks.
        \item We use the four cardinal directions, north, south, east and west to
        to indicate where one point lies in relation to another in the visual representation of the permutation matrix $M(\pi)$.
        Explicitly, a point is \textit{north} (respectively, \textit{south}) of another point if and only if the row index in $M(\pi)$ of the former point is smaller than that of the latter,
        and is \textit{west} (respectively, \textit{east}) of another point if and only if the column index in $M(\pi)$ of the former point is smaller than that of the latter.
        \item For $i_1,i_2 \in [m+1]$, we say that $\alpha_{i_1j}$ is \textit{to the left} (respectively, \textit{to the right}) of $\alpha_{i_2j}$
        if and only if all the points in $\alpha_{i_1j}$ are west (respectively, east) of all the points in $\alpha_{i_2j}$.
        \item For $j_1,j_2 \in [m+1]$, we say that $\alpha_{ij_1}$ is \textit{superior} (respectively, \textit{inferior} to $\alpha_{ij_2}$
        if and only if all nontrivial columns of $\alpha_{ij_2}$ are north (respectively, south) of all nontrivial columns of $\alpha_{ij_2}$.
        \item We say that $\alpha_{ij}$ is \textit{leftmost} (respectively, \textit{rightmost})
        if and only if the first (respectively, last) column of $\alpha_{ij}$ is nontrivial.
        \item We say that $\alpha_{ij}$ is \textit{topmost} (respectively, \textit{bottommost})
        if and only if the first (respectively, last) row of $\alpha_{ij}$ is nontrivial.
    \end{enumerate}
\end{definition}

\section{Permutations avoiding \texorpdfstring{$\lambda$}{lambda}}\label{chap:av1}

Gao and Kitaev \cite{gao-kitaev-2019} observed that there the POP $\lambda$ of size 4 where $1>2$ and $1>4$ (illustrated in Figure \ref{fig:lambda})
and the set of patterns \[\mathcal{P}=\{2413, 2431, 4213, 3412, 3421, 4231, 4321, 4312\}\] are Wilf-equivalent.
This was done by proving the recursive equation \[\abs{Av_n(\lambda)} = 3\abs{Av_{n-1}(\lambda)}-2\abs{Av_{n-2}(\lambda)}+2\abs{Av_{n-3}(\lambda)},\]
which corresponds to the OEIS sequence \href{https://oeis.org/A111281}{A111281}.
In this section, we will give a new proof that $\lambda$ and $\mathcal{P}$ are Wilf-equivalent
as well as provide a new recursive formula for the OEIS sequence.
We also analyse each avoidance set in detail and construct an explicit bijection between them.

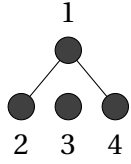
\begin{figure}[!htbp]
    \centering
    \begin{tikzpicture}
	\begin{pgfonlayer}{nodelayer}
		\node [style=node] (0) at (0, 1.5) {};
		\node [style=node] (1) at (-1.25, 0) {};
		\node [style=node] (2) at (1.25, 0) {};
		\node [style=none] (4) at (0, 2.5) {1};
		\node [style=none] (5) at (-1.25, -1) {2};
		\node [style=none] (6) at (0, -1) {3};
		\node [style=none] (7) at (1.25, -1) {4};
		\node [style=node] (8) at (0, 0) {};
	\end{pgfonlayer}
	\begin{pgfonlayer}{edgelayer}
		\draw (0.center) to (1.center);
		\draw (0.center) to (2.center);
	\end{pgfonlayer}
\end{tikzpicture}
    \caption{The POP $\lambda$}
    \label{fig:lambda}
\end{figure}

First, we show that both avoidance sets have only finitely many simple permutations.
Using this fact, we analyse the possible inflations of these simple permutations in each avoidance set
and derive a method to construct each set recursively.
A recursively-defined bijection on the two sets follows immediately from this analysis.\\

In this section, we use the symbol $I_k$ to denote the identity permutation $1\, 2\, 3\, \cdots \, k$ for all $k\geq 1$.

\subsection{Structure of permutations avoiding \texorpdfstring{$\lambda$}{lambda}}

\begin{lemma}\label{Av1(lambda):simple}
    The only simple permutations that avoid $\lambda$ are 12, 21 and 2413.
\end{lemma}
\begin{proof}
    It is clear that 2413 and all simple permutations of length 3 or less avoid $\lambda$.
    Consider the permutation matrix of a simple permutation $\pi$ with length at least $4$.
    It must contain the pattern $312$ by Corollary \ref{cor:separable},
    say $1\leq i<j<k\leq n$ where $\text{red}(\pi_i \pi_j \pi_k)=312$.
    We can then consider the lattice matrix $L_{312}(\pi)$ which is illustrated in Figure \ref{fig:lattice}.
    It suffices to prove that $\alpha_{31}$ has weight 1,
    while the other blocks are trivial.\\

    Upon inspection, it is clear that
    if any of $\alpha_{11},\,\alpha_{13}$ or $\alpha_{ij}$, where $i,\, j\in [2,4]$, were not trivial,
    then the permutation would contain $\lambda$.
    For the reader's convenience,
    we reproduce the figure with those $\alpha_{ij}$'s omitted in Figure \ref{fig:Av1(lambda)}.
    We now proceed to show that the remaining blocks,
    except for $\alpha_{31}$, are trivial:

    \begin{figure}[!htbp]
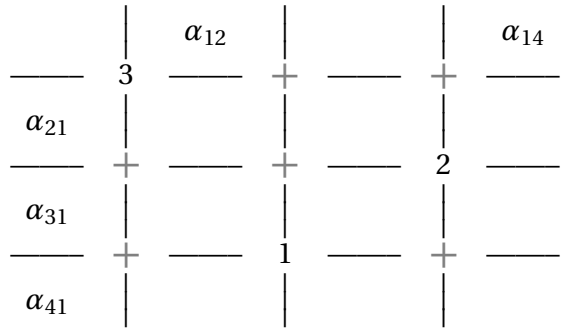

        \begin{center}
            \begin{tabular}{ccccccc}
                              &   \Big |  & $\alpha_{12}$ &   \Big |  &          &   \Big |  & $\alpha_{14}$\\
                --------      &     3     &    --------   & $\Plus$   & -------- & $\Plus$   & --------\\
                $\alpha_{21}$ &   \Big |  &               &   \Big |  &          &   \Big |  & \\
                --------      & $\Plus$   &    --------   & $\Plus$   & -------- &      2    & --------\\
                $\alpha_{31}$ &   \Big |  &               &   \Big |  &          &   \Big |  & \\
                --------      & $\Plus$   &    --------   &      1    & -------- & $\Plus$   & --------\\
                $\alpha_{41}$ &   \Big |  &               &   \Big |  &          &   \Big |  & \\
            \end{tabular}
            \caption{The lattice matrix $L_{312}(\pi)$, with some blocks omitted. The omitted blocks must be trivial for $\pi$ to avoid $\lambda$.}
            \label{fig:Av1(lambda)}
        \end{center}
    \end{figure}

    \begin{enumerate}[(a)]
        \item Suppose $\alpha_{41}$ is not trivial.
        It cannot be leftmost,
        since otherwise the permutation would be sum decomposable.
        However, if $\alpha_{21}$ or $\alpha_{31}$ contains a point
        east of a point in $\alpha_{41}$,
        then $\pi$ would contain $\lambda$
        (consider the 1 in $\alpha_{21}$ or $\alpha_{31}$, together with the 1 in $\alpha_{41}$ and the points labelled 3 and 1).
        So $\alpha_{41}$ is trivial.

        \item Suppose $\alpha_{21}$ is not trivial.
        Since it is adjacent to the block labelled 3,
        it cannot be rightmost in its column by simpleness.
        However, if $\alpha_{31}$ contains a point
        east of a point of $ \alpha_{21}$,
        then $\pi$ would contain $\lambda$
        (consider the 1s in $\alpha_{21}\alpha_{31}$, together with the points labelled 3 and 1).
        So $\alpha_{21}$ is trivial.

        \item Suppose $\alpha_{14}$ is not trivial.
        If $\alpha_{14}$ is superior to $\alpha_{12}$,
        then the permutation would be sum decomposable.
        So $\alpha_{12}$ must contain a nontrivial row superior to
        a nontrivial row of $ \alpha_{12}$.
        However, $\pi$ would contain then $\lambda$
        (consider the 1s in $\alpha_{12}\alpha_{14}$, together with the points labelled 1 and 2).
        So $\alpha_{14}$ is trivial.

        \item Observe that $\alpha_{12}$ is adjacent to the block labelled 3.
        Since the blocks in the same row or column as $\alpha_{12}$ are trivial,
        $\alpha_{12}$ must also be trivial.
    \end{enumerate}

    Finally, since all the blocks in the same row or column as $\alpha_{31}$ are trivial,
    $\alpha_{31}$ can have weight at most 1 by simpleness.
    We have thus eliminated the possibility of there being a simple permutation of length at least 4 avoiding $\lambda$ that is not 2413,
    so our list is exhaustive.
\end{proof}

\begin{lemma}\label{Av1(lambda):21}
    For $n\geq 4$, there are $n$ skew sum decomposable permutations in $Av_n(\lambda)$,
    namely $21[I_{n-1},12]$, $21[I_{n-2},21]$ and
    $2431[I_{\ell},\, I_{n-\ell-2},\, 1,\, 1]$ where $\ell\in [2,n-1]$.
\end{lemma}

\begin{proof}
    Let $\pi=21[\alpha,\beta]$ be an $n$-permutation avoiding $\lambda$.
    It is not hard to see that if $\abs{\beta}\geq 3$, then $\pi$ contains $\lambda$.
    So $\abs{\beta}=1$ or 2:
    \begin{enumerate}[1)]
        \item Suppose $\abs{\beta}=2$. If $\alpha$ contains a descent,
        then the elements that make up the descent,
        together with $\beta$ make up $\lambda$.
        So $\alpha$ must be an increasing sequence,
        and indeed both
        $21[I_{n-1},12]$ and $21[I_{n-2},21]$ avoid $\lambda$.
        \item Suppose $\abs{\beta}=1$. Then $\pi$ avoids $\lambda$ if and only if $\alpha  $ avoids the POP of size 3 where $1 > 2$.
        This is exactly when the first $n-2$ elements of $\alpha$ are increasing.
        Since $\alpha$ is assumed to be skew sum indecomposable (for uniqueness), the last element of $\alpha  $ cannot be $1$.
        Therefore there are $n-2$ choices for the last element of $\alpha$,
        and only one way to order the initial elements.
        Indeed, for all $2\leq \ell\leq n-1$,
        the following permutation avoids $\lambda$:
        \begin{align*}
            \pi
            &=21[12\cdots \ell (\ell+2)\cdots (n-2)(\ell+1),1]\\
            &=21[132[I_{\ell},\, I_{n-\ell-2},1],1]\\
            &=2431[I_{\ell},\, I_{n-\ell-2},\, 1,\, 1]
        \end{align*}
    \end{enumerate}
    Therefore, there are $(n-2)+2=n$ skew sum decomposable $n$-permutations avoiding $\lambda$ in total.
\end{proof}

\begin{lemma}\label{Av1(lambda):2413}
    For $n\geq 4$, there are $n-3$ that are inflations of 2413 avoiding $\lambda$ that are of length $n$.
    Specifically, they are of the form $2413[I_{\ell}, \, I_{n-\ell-2},\,1,\,1]$ for $\ell \in [n-3]$.
\end{lemma}

\begin{proof}
    Let $\pi:=2413[\alpha,\beta,\gamma,\delta]$ be an $n$-permutation avoiding $\lambda$.
    We will show that $\abs{\gamma}=\abs{\delta}=1$,
    while $\alpha$ and $\beta$ are increasing sequences but can be of variable length.\\

    \noindent Suppose $\abs{\gamma}\geq 2$ or $\abs{\delta}\geq 2$.
    Then $\pi$ contains $\lambda$ (consider one point from $\beta$) and three points total from $\gamma$ and $\delta$.
    Now suppose that $\alpha$ or $\beta$ contains a descent.
    Then the two elements that make up the descent,
    together with one element from $\gamma$ and one element from $\delta$ make $\lambda$.
    Therefore, $\abs{\gamma}=\abs{\delta}=1$ and
    $\alpha$ and $\beta$ are increasing.\\

    \noindent Finally, it is not hard to see that for all $\ell\in [n-3]$,
    the permutation $2413[I_{\ell}, \, I_{n-\ell-2},\,1,\,1]$
    avoids $\lambda$.
    So there are $n-3$ inflations of 2413 avoiding $\lambda$.
\end{proof}

\begin{theorem}\label{theorem:Av1(lambda)}
    For all $n\geq 4$, $\ds \abs{Av_n(\lambda)} = 2n-3+\sum_{i=1}^{n-1}(2i-3) \abs{Av_{n-i}(\lambda)}$.
\end{theorem}
\begin{proof}
    Lemmas \ref{Av1(lambda):simple}, \ref{Av1(lambda):21} and \ref{Av1(lambda):2413}
    together imply that
    there are $2n-3$ sum indecomposable permutations in $Av_n(\lambda)$.
    Moreover, a sum decomposable permutation $12[\alpha,\beta]$ avoids $\lambda$
    if and only if $\alpha$ and $\beta$ both avoid $\lambda$.
    Therefore, there are  $\ds\sum_{i=1}^{n-1}(2i-3) \abs{Av_{n-i}(\lambda)}$ sum decomposable permutations
    of the form $12[\alpha,\beta]$ in $Av_n(\lambda)$, where $\alpha$ is sum indecomposable.
    Thus we get the recursive formula for $Av_n(\lambda)$.
\end{proof}

\subsection{Structure of permutations avoiding \texorpdfstring{$\mathcal{P}$}{P}}

\noindent Next, we demonstrate the $n$-permutations of $Av(\mathcal{P})$ explicitly.
Recall that \[\mathcal{P}=\{2413, \,2431, \,4213, \,3412, \,3421,\, 4231, \,4321, \,4312\}.\]

\begin{lemma}\label{Av1(P):simple}
The only simple permutations that avoid $\mathcal{P}$ are 12, 21, 3142 and 41352.
\end{lemma}
\begin{proof}
    It is clear that 3142, 41352 and all simple permutations of length $3$ or less avoid $\mathcal{P}$.
    Consider a simple permutation $\pi$ of length at least $4$ avoiding $\mathcal{P}$ .
    Since $\mathcal{P}$ contains 2413, $\pi$ avoids 2413 and must contain $3142$,
    since otherwise it would be a separable permutation and not simple by Theorem \ref{thm:separable}.\\

    We present its lattice matrix $L_{3142}(\pi)$ in Figure \ref{Av1(P)},
    with some alterations explained in the caption.
    It suffices to show that $\alpha_{33}$ can have weight at most 1,
    while the remaining $\alpha_{ij}$ are trivial:

    \begin{figure}[!htbp]
        \begin{center}
            \begin{tabular}{ccccccccc}
                4312          & \big |  &      3412     & \big |  &     2431      & \big |  & $\alpha_{14}$ & \big |  & $\alpha_{15}$\\
                ------        & $\Plus$ &     ------    & $\Plus$ &     ------    &   4     &     ------    & $\Plus$ & ------ \\
                4312          & \big |  &      3412     & \big |  & $\alpha_{23}$ & \big |  &     2431      & \big |  & 2413\\
                ------        &   3     &     ------    & $\Plus$ &     ------    & $\Plus$ &     ------    & $\Plus$ & ------ \\
                3412          & \big |  &      4312     & \big |  & $\alpha_{33}$ & \big |  &     3421      & \big |  & 3412\\
                ------        & $\Plus$ &     ------    & $\Plus$ &     ------    & $\Plus$ &     ------    &    2    & ------ \\
                2413          & \big |  &      4231     & \big |  & $\alpha_{43}$ & \big |  &     3412      & \big |  & 3421\\
                ------        & $\Plus$ &     ------    &   1     &     ------    & $\Plus$ &     ------    & $\Plus$ & ------ \\
                $\alpha_{51}$ & \big |  & $\alpha_{52}$ & \big |  &     4213      & \big |  &     3412      & \big |  & 3421
            \end{tabular}
            \caption{The lattice matrix $L_{3142}(\pi)$ with some $\alpha_{ij}$ replaced by a pattern in $\mathcal{P}$
            that $\pi$ would contain if that $\alpha_{ij}$ were nontrivial, for $i,j\in [5]$.
            For example, if $\alpha_{12}$ were nontrivial, then $\pi$ would contain 3412.}
            \label{Av1(P)}
        \end{center}
    \end{figure}

    We proceed to show that $\alpha_{ij}$ must trivial for all $i,j\in [5]$, except for $i=j=3$.
    \begin{enumerate}[(a)]
        \item Suppose $\alpha_{52}$ were nontrivial.
        Since it is adjacent to the point 1,
        the block $\alpha_{51}$ cannot be trivial and must be superior to $\alpha_{52}$ by simpleness.
        However, this would mean the inclusion of the pattern $2413$
        (consider any submatrix containing $\alpha_{51}\, 3\, \alpha_{52}\, 1$).

        \item Since $\alpha_{5j}$ are trivial for all $j\in [2,5]$,
        the block $\alpha_{51}$ must be trivial as well,
        for otherwise $\pi$ would be sum decomposable.

        \item Suppose $\alpha_{14}$ were nontrivial.
        Since it is adjacent to the point 4,
        the block $\alpha_{15}$ cannot be trivial and must be inferior to $\alpha_{14}$ by simpleness.
        However, this would mean the inclusion of the pattern $2413$
        (consider any submatrix containing $4\, \alpha_{14}\, 2\, \alpha_{15}$).

        \item Suppose $\alpha_{23}$ were nontrivial.
        Since it is adjacent to the point 4,
        it cannot be rightmost by simpleness.
        However, this would mean the inclusion of the pattern $3412$
        (consider any submatrix containing $3\, \alpha_{23}\, \alpha_{33}\, 2$ or $3\, \alpha_{23}\, \alpha_{43} \, 2$).

        \item Suppose $\alpha_{43}$ were nontrivial.
        Since it is adjacent to the point 1,
        it cannot be leftmost by simpleness.
        Then $\alpha_{33}$ must be nontrivial and lie to the left of $\alpha_{43}$.
        However, this would mean the inclusion of the pattern $4312$
        (consider any submatrix containing $3\, \alpha_{33} \,\alpha_{43}\, 2$).
    \end{enumerate}

    Since all the blocks in the same row or column as $\alpha_{33}$ are trivial,
    it can have weight at most 1 by simpleness.
    Moreover, it cannot be trivial since all the other $\alpha_{ij}$s are trivial and 312 is not simple.
    We have thus eliminated the possibility of there being a simple permutation of length at least 5 avoiding $\mathcal{P}$ that is not 41352,
    so our list is exhaustive.
\end{proof}

\begin{lemma}\label{Av1(P):21}
    There are $4$ skew sum decomposable $n$-permutations in $Av(\mathcal{P})$.
    Namely, they are $21[1, \, I_{n-1}]$, $312[1,\, I_{n-3}, \, 21]$, $21[I_{n-1},\, 1]$ and $231[21,\, I_{n-3},\, 1]$.
\end{lemma}
\begin{proof}
    Let $\pi=21[\alpha,\beta]$ be a permutation avoiding $\mathcal{P}$. We have 3 cases:
    \begin{enumerate}[1)]
        \item If $\abs{\alpha},\abs{\beta}\geq 2$, then $\pi$ contains $3412,\,3421,\,4312$ or $4321$.

        \item Suppose $\abs{\alpha}=1$ and $\abs{\beta}\geq 2$.
        Then $\pi$ avoids $\mathcal{P}$
        if and only if $\beta$ avoids $213,\, 231,\, 321,\, 312$. \\
        That is, $\pi$ avoids $\mathcal{P}$
        if and only if all but the last two terms of $\beta$ are strictly increasing.\\
        There are only two such permutations, namely
        $21[1, \, I_{n-1}]$ and $312[1,\, I_{n-3}, \, 21]$.

        \item Suppose $\abs{\beta}=1$ and $\abs{\alpha}\geq 2$.
        Then $\pi$ avoids $\mathcal{P}$ if and only if $\alpha$ avoids
        \[\text{red}(243)=132, \quad \text{red}(342)=231,\quad \text{red}(423)=312 \quad \text{and}\quad \text{red}(432)=321.\]
        That is, if and only if all but the first two terms of $\alpha$ are strictly increasing.\\
        There are only two such permutations, namely
        $21[I_{n-1},\, 1]$ and $231[21,\, I_{n-3},\, 1]$.
    \end{enumerate}
\end{proof}

\begin{lemma}\label{Av1(P):3142}
    There are $n-3$ inflations of 3142 of length $n$ avoiding $\mathcal{P}$.
    Specifically, they are of the form $3142[1,\, I_{\ell}, \, I_{n-\ell-2},\,1]$ where $\ell\in[n-3]$.
\end{lemma}
\begin{proof}
    Let $\pi=3142[\alpha,\beta,\gamma,\delta]$ be a permutation avoiding $\mathcal{P}$.
    We will show that $\abs{\alpha}=\abs{\delta}=1$,
    while $\beta$ and $\gamma$ are increasing sequences of variable length.
    \begin{enumerate}
        \item If $\alpha$ contains an ascent, then $\pi$ contains 3412
        (consider $312[\alpha,\beta,\delta]$).\\
        On the other hand, if $\alpha$ contains an descent, then $\pi$ contains 4312
        (consider $312[\alpha,\beta,\delta]$).
        \item If $\delta$ contains an ascent, then $\pi$ contains 3412 (consider $342[]\alpha,\gamma,\delta]$).\\
        On the other hand, if $\delta$ contains an descent, then $\pi$ contains 3421 (consider the subpermutation $342[\alpha,\gamma,\delta]$).
        \item If $\beta$ contains an descent, then $\pi$ contains 4213 (consider $314[\alpha,\beta,\delta]$).\\
        If $\gamma$ contains an descent, then $\pi$ contains 2431 (consider $342[\alpha,\gamma,\delta]$).
        Therefore $\beta$ and $\gamma$ are increasing.
    \end{enumerate}

    \noindent Finally, it is not hard to see that for all $\ell\in[n-3]$,
    the permutation $3142[1,\, I_{\ell}, \, I_{n-\ell-2},\,1]$
    avoids $\lambda$.
    Therefore, there are $n-3$ inflations of 3142 in $Av_n(\mathcal{P})$.
\end{proof}

\break

\begin{lemma}\label{Av1(P):41352}
    There are $n-4$ inflations of 41352 of length $n$ avoiding $\mathcal{P}$.
    Specifically, they are of the form $41352[1, I_{\ell},1,I_{n-\ell-3},1]$ where $\ell \in [n-4]$.
\end{lemma}

\begin{proof}
    Let $41352[\alpha,\beta,\gamma,\delta,\zeta]$ be a permutation avoiding $\mathcal{P}$.
    We will show that
    $\abs{\alpha}=\abs{\gamma}=\abs{\zeta}=1$,
    while $\beta$ and $\delta$ are both increasing sequences of variable length.
    \begin{enumerate}
        \item If $\alpha$ contains an ascent, then $\pi$ contains 3412 (consider $413[\alpha,\beta,\gamma]$).\\
        If $\alpha$ contains an descent, then $\pi$ contains 4312 (consider $413[\alpha\beta,\gamma]$).\\
        Therefore $\abs{\alpha}=1$.
        \item If $\gamma$ contains an ascent, then $\pi$ contains 4231 (consider $432[\alpha,\gamma,\zeta]$).\\
        If $\gamma$ contains an descent, then $\pi$ contains 4231 (consider $432[\alpha,\gamma,\zeta]$).\\
        Therefore $\abs{\gamma}=1$.
        \item If $\zeta$ contains an ascent, then $\pi$ contains 4312 (consider $432[\alpha,\gamma,\zeta]$).\\
        If $\zeta$ contains an descent, then $\pi$ contains 4321 (consider $432[\alpha,\gamma,\zeta]$).\\
        Therefore $\abs{\zeta}=1$.
        \item If $\beta$ contains an descent, then $\pi$ contains 4213 (consider $412[\alpha,\beta,\zeta]$).\\
        Therefore $\beta$ is increasing.
        \item If $\delta$ contains an descent, then $\pi$ contains 2431 (consider $452[\alpha,\delta,\zeta]$).\\
        Therefore $\delta$ is increasing.
    \end{enumerate}

    \noindent Finally, it is not hard to see that for all $\ell \in [n-4]$,
    the permutation $41352[1, I_{\ell},1,I_{n-\ell-3},1]$
    avoids $\mathcal{P}$.
    Therefore, there are $n-4$ inflations of 41352 in $Av_n(\mathcal{P})$.
\end{proof}

\begin{theorem}\label{theorem:Av1(P)}
    For all $n\geq 4$, $\ds  \abs{Av_{n}(\mathcal{P})}= 2n-3+\sum_{i=1}^{n-1}(2i-3) \abs{Av_{n-i}(\mathcal{P})}$.
\end{theorem}
\begin{proof}
    From Lemma \ref{Av1(P):simple}, \ref{Av1(P):21}, \ref{Av1(P):3142} and \ref{Av1(P):41352},
    we can easily see that
    there are $2n-3$ sum indecomposable $n$-permutations in $Av(\mathcal{P})$.
    Moreover, a sum decomposable $n$-permutation $12[\alpha,\beta]$ avoids $\mathcal{P}$
    if and only if $\alpha$ and $\beta$ both avoid $\mathcal{P}$,
    so there are $\ds\sum_{i=1}^{n-1}(2i-3)  \abs{Av_{n-i}(\mathcal{P})}$ permutations of the form $12[\alpha,\beta]$
    in $Av_n(\mathcal{P})$ where $\alpha$ is sum indecomposable.
\end{proof}


\subsection{The bijection}
Having thoroughly analysed the structure of the two avoidance sets, we are ready to construct the bijection explicitly:

\begin{theorem}
    Define $g$ to be a function that maps
    \begin{align*}
        1&\mapsto 1\\
        321&\mapsto 321\\
        312&\mapsto 231\\
        21[I_{k},\, 1] &\mapsto 21[1,\, I_{k}]\\
        231[I_{k},\, 21,\, 1] &\mapsto 312[1, \, I_{k},\, 21]\\
        21[I_{k+1},\, 21] &\mapsto 231[21,\, I_{k},\, 1]\\
        2431[I_{k},\, I_{j+1},\, 1,\, 1] &\mapsto 41352[1,\, I_{k}, \, 1, \, I_{j},\, 1]\\
        2413[I_{k},\, I_{j},\, 1,\, 1] &\mapsto 3142[1,\, I_{k},\,  I_{j},\, 1]
    \end{align*}
    for all $k,j\geq 1$.
    Define $f:Av(\lambda)\rightarrow Av(\mathcal{P})$ by
    \begin{align*}
        f(\pi) = \begin{cases}
            g(\pi) &\text{if }\pi \text{ is sum indecomposable}\\
            12[g(\alpha),f(\beta)] &\text{if }\pi=12[\alpha,\beta], \text{ and }\alpha\text{ is a sum indecomposable}.
        \end{cases}
    \end{align*}
    Then $f$ is a bijection.
    In fact, $f$ restricted to $Av_n(\lambda)$ is a bijection onto $Av_n(\mathcal{P})$.
\end{theorem}

\begin{proof}
    From the lemmas in the previous two sections, it is clear that $g$ maps the
    sum indecomposable permutations in $Av_n(\lambda)$ to the sum indecomposable permutations in $Av_n(\mathcal{P})$.
    Moreover $f$ maps permutations in $Av_n(\lambda)$ into $Av_n(\mathcal{P})$
    by the proofs of Theorem \ref{theorem:Av1(lambda)} and Theorem \ref{theorem:Av1(P)},
    and it is clear that $f$ is an injection.
    Since $\abs{Av_n(\lambda)}=\abs{Av_n(\lambda)}$,
    $f$ restricted to $Av_n(\lambda)$ is a bijection onto $Av_n(\mathcal{P})$ for all $n\geq 1$.
    Thus $f$ is a bijection from $Av(\lambda)$ to $Av(\mathcal{P})$.
\end{proof}

\section{\texorpdfstring{$Q_k$}{Qk} and two connections to classical combinatorial objects}\label{chap:Q}

\begin{definition}
    Let $Q_k$ be the POP with $k$ elements where $1>t$ for $t\in [2,n]$.
    The Hasse diagram of $Q_k$ is the following:
    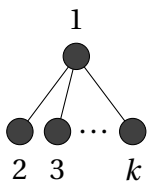
\begin{figure}[!htbp]
        \centering
        \begin{tikzpicture}
	\begin{pgfonlayer}{nodelayer}
		\node [style=node] (0) at (0, 2) {};
		\node [style=node] (1) at (-1.5, 0) {};
		\node [style=node] (2) at (-0.5, 0) {};
		\node [style=node] (3) at (1.5, 0) {};
		\node [style=none] (4) at (0.5, 0) {$\dots$};
		\node [style=none] (5) at (0, 3) {1};
		\node [style=none] (6) at (-1.5, -1) {2};
		\node [style=none] (7) at (-0.5, -1) {3};
		\node [style=none] (8) at (1.5, -1) {$k$};
	\end{pgfonlayer}
	\begin{pgfonlayer}{edgelayer}
		\draw (0.center) to (1.center);
		\draw (0.center) to (2.center);
		\draw (0.center) to (3.center);
	\end{pgfonlayer}
\end{tikzpicture}
        \caption{The POP $Q_k$}
        \label{fig:Qk}
    \end{figure}
\end{definition}

Gao and Kitaev \cite{gao-kitaev-2019} enumerated the avoidance set of $Q_k$ in Theorem 2 of their paper
and observed that the sequences $\abs{Av_n(Q_4)}_{n\geq 1}$ and $\abs{Av_n(Q_5)}_{n\geq 1}$
are listed in the OEIS database as \href{https://oeis.org/A025192}{A025192} and \href{https://oeis.org/A084509}{A084509} respectively.
These sequences enumerate the 2-ary shrub forests of $n$ heaps avoiding the patterns 231, 312 and 321,
and the ground state 3-ball juggling sequence of period $n$ respectively.
These objects will be defined in the following sections.\\

We show that there are intimate connections between the relevant sets by constructing explicit bijections.
In fact, we construct an explicit bijection from $Av_n(Q_{k})$ to ground state $(k-2)$-ball juggling sequence of period $n$
which yields a bijection between the original two sets as a special case.
The study that led us to this bijection also produced a new way of enumerating $Q_k$ using the concept of matrix permanents.


\subsection{Enumeration of \texorpdfstring{$Av_n(Q_k)$}{Avn(Qk)}}

We will enumerate $Av_n(Q_k)$ for $k,\, n\geq 1$ using the concept of matrix permanents.
Recall that the permanent is defined on square matrices and is similar to the definition of the determinant of a matrix,
where the signs of the summands are all positive instead of alternating.
We restate Gao and Kitaev's proof first as a reference for comparison.

\begin{theorem}[Gao and Kitaev (2019) \cite{gao-kitaev-2019}]\label{thm:avQk}
    For $n\geq k$, $\abs{Av_n(Q_k)} = (k-1)!\times (k-1)^{n-k+1}$.
\end{theorem}
\begin{proof}
    We proceed by induction.
    It is clear that all $(k-1)$-permutations avoid $Q_k$,
    so $\abs{Av_{k-1}(Q_k)}=(k-1)!$.
    For $n\geq k$, $\pi$ is an $n$-permutation avoiding $Q_k$
    if and only if $n\in \pi_{[n-k+2,\, n]}$
    and the $(n-1)$-permutation obtained from removing $n$ from the permutation $\pi$ avoids $Q_k$.
    So $\abs{Av_{n}(Q_k)} = (k-1)\times \abs{Av_n (Q_{k-1})}$.
    The desired formula can then be easily obtained from this recursion via induction on $k$.
\end{proof}

\begin{definition}
    The \textit{permanent} of an $n\times n$ matrix $A_n=(a_{ij})_{i,j\in [n]}$ is
    \[\text{Perm}(A_n):= \sum_{\pi\in S_n} \prod_{i=1}^n  a_{i j}.\]
\end{definition}



The following proposition states a well-known property of permanents.
Its proof is similar to the proof for an analogous statement for determinants
and is therefore omitted.

\begin{prop}[\textit{folklore}]
    The permanent of a matrix is invariant under arbitrary permutations of the rows and/or columns,
    as well as transposition.
    That is, for all $n\times n$ matrices $M$ and $n$-permutation matrices $P$ and $Q$,
    \begin{center}
        Perm$(M^T)=$ Perm$(M)=$ Perm$(PMQ)$.
    \end{center}
\end{prop}

\bigskip
Observe that the permanent of the $n\times n$ matrix of ones is equal to $n!$ for all positive $n$,
and recall that there are $n!$ permutations in $S_n$.
One may ask whether there is a matrix associated with subsets of $S_n$,
such that the problem of enumerating those subsets can be converted into
computing a certain value of a matrix.
This is in fact possible for certain subsets of $S_n$, as we shall see.

\begin{definition}
    Let $n\geq 1$.
    Suppose $K$ is a set of restrictions that indicate whether
    $i\mapsto j$ in an $n$-permutation is allowed for all $i,\, j\in [n]$.
    Then define $A_n^K=(a_{ij})$ as the binary $n\times n$ matrix
    where, for all $i,\, j \in [n]$,
    the $ij$th element of $A_n^K$ is denoted $a_{ij}$ and is equal to 1 if and only if
    having $i\mapsto j$ is allowed by $K$.
    We say that the matrix $A_n^K$ \textit{represents} $K$.
\end{definition}

\bigskip

\begin{lemma}[Percus (1971) \cite{percus}]\label{lemma:perm}
    Given a set of restrictions $K_n$ and matrix $A_n^K$ defined as above,
    the number of $n$-permutations that satisfy $K$ is the permanent of $A_n^K$.
\end{lemma}
\begin{proof}
    Let $f$ be the function from $[n]$ to the set of subsets of $[n]$ such that $i\rightarrow \pi_i$ is allowed in a permutation if and only if $\pi_i\in f(i)$.
    Let $a_{ij}$ denote the $ij$th entry of $A_n^K$.
    Then the number of $n$-permutations induced by $f$ is
    \begin{align*}
        \#\{\pi\in S_n \mid \pi_i\in f(i) \fa i\in [1,n]\}
        &=\#\{\pi\in S_n \mid a_{i \pi_i}=1 \fa i\in [1,n]\}\\
        &= \sum_{\pi\in S_n} a_{1 \pi_1} a_{2 \pi_2} \cdots a_{n \pi_n} \\
        &= \sum_{\pi\in S_n} \prod_{i=1}^n  a_{i \pi_i} \\
        &= \text{Perm}(A).
    \end{align*}
\end{proof}

We can now apply the theorem to the enumeration of the avoidance set of $Q_k$ for all $k\geq 1$:
Observe that for all $n$, an $n$-permutation $\pi$ avoids $Q_k$ if and only if
$t\in \pi_{[t-k+2,n]}$ for all $t\in [k,n]$.
Therefore we have the following theorem:

\begin{theorem}\label{thm:Q-matrix}
    The $n$-permutations that avoid $Q_k$ are represented by the binary $n\times n$ matrix
    $A_n=(a_{ij})_{i,j\in [n]}$ where $a_{ij}=1$ if and only if $i\geq j-k+2$.
\end{theorem}

Since the permanent of $A_n$
is exactly $(k-1)! (k-1)^{n-k+1}$ for all $n\geq k$,
the theorem above proves the enumeration of the avoidance set of $R_k$ for all $n$.

\subsection{Juggling sequences}\label{sect:juggling}
Suppose we have $b$ balls and a binary vector $\sigma=(\sigma_1,\sigma_2,\dots,\sigma_n)\in \{0,1\}^n$ for some integer $n\geq b$.
Given a reference time,
we can throw one ball at the $i$th second
such that it lands in our hand again after
exactly $t_i$ seconds (that is, $i+t_i$ seconds after the reference time)
where $t_i$ is a positive integer
for all $i\in [n]$, if $\sigma_i=1$.
Otherwise, we put $t_i:=0$.
We say that the $n$-tuple of non-negative integers $T=(t_1,t_2,\dots,t_n)$ is
a \textit{juggling sequence of period $n$ and \textit{state} $\sigma$}
if and only if
no two balls land in our hand at the same time,
and our sequence of throws is infinitely repeatable.
That is, the following two conditions hold:
\begin{itemize}
    \item $i+t_i\equiv j+t_j\pmod{n}$ if and only if $i=j$ for all $i,j\in [n]$,
    \item $\sigma_i=1$ if and only if there is some $j\in [n]$ such that $i\equiv j+t_j\pmod{n}$.
\end{itemize}
We say that $\sigma$ is a \textit{ground state} if and only if
$\sigma_i=1$ if $i\in [b]$ and $\sigma_i=0$ otherwise.
We refer the reader to \cite{juggling-drops} and \cite{chung-graham} for diagrams and further analysis on juggling sequences.\\

Gao and Kitaev \cite{gao-kitaev-2019} observed that the avoidance set $Av_n(Q_5)$ and the
number of ground-state 3-ball juggling sequences of period $n$
are enumerated by the same OEIS sequence \href{https://oeis.org/A084509}{A084509}.
We will prove a natural bijection between a generalization of these two combinatorial objects,
which is inspired by the proof of Theorem 1 of \cite{chung-graham}, restated below:

\begin{theorem}\label{thm:juggling-enum}
    The number of ground state juggling sequences of period $n$ using $b$ balls ($n\geq b$) is
    \[J(n,b)=\begin{cases}
        (b+1)^{n-b} b!&\text{if }n\geq b,\\
        n!&\text{otherwise}.
    \end{cases}\]
\end{theorem}

\begin{proof}
    Observe that $T=(t_1,t_2,\dots,t_n)$ is a conforming juggling sequence if and only if
    every ball that is thrown lands exactly at some time $t$ seconds after the reference time
    where $t\in [n+1,n+b]$, and no two balls land on the same second.
    That is, $\{t_i+i\mid i\in [b]\} = [n+1,n+b]$.
    Since $t_i=0$ for $i\in [n]\setminus [b]$,
    the condition is equivalent to requiring
    $\{t_i+i-b\mid i\in [n]\} = [n]$.
    That is, the bijection $i\mapsto t_i+i-b$ defines a permutation on the set $[n]$,
    with the additional condition that $t_i\geq 0$ for all $i\in [n]$.
    Thus every conforming juggling sequence $T$ can be associated uniquely to a permutation $\pi$
    where $\pi_i=t_i+i-b$.
    The number of such permutations is then exactly the permanent of the matrix
    $M$ where its $ij$th entry is 1 if and only if $j-i+b\geq 0$, by Lemma \ref{lemma:perm}.
    Computing the permanent of $M$ yields the formula assigned to $J(n,b)$.
\end{proof}

\begin{table}[!htbp]
    \centering
    \begin{tabular}{|c|c|c|c|c|}
        \hline
        \thead{State,\\ period,\\ \# balls,\\\# juggling\\sequences} & \thead{Juggling\\sequences\\$T=(t_1,\dots,t_n)$} & \thead{$(t_1+1,\dots,t_n+n)$\\$=(\sigma_1,\dots,\sigma_n)$\\$+(b,\dots,b)$} & \thead{$\sigma\inv=\pi$\\$\in Av_n(Q_{b+2})$;\\$\sigma_i = t_i+i-b$} & \thead{Matrix\\transpose\\$M^t=M_{n,b}^t$} \\\hline
        \makecell{$\sigma=(1,0)$;\\$n=2$;\\ $b=1$;\\$J(n,b)=2$}      & \makecell{(1,1)\\(2,0)}  & \makecell{(2,3)\\(3,2)} & \makecell{$12\inv=12$\\$21\inv=21$} & $\left(\begin{array}{cc} 1&1\\1&1 \end{array}\right)$  \\
        \makecell{$\sigma=(1,0,0)$;\\$n=3$;\\ $b=1$;\\$J(n,b)=4$}    & \makecell{(1,1,1)\\(3,0,0)\\(2,0,1)\\(1,2,0)}  & \makecell{(2,3,4)\\(4,2,3)\\(3,2,4)\\(2,4,3)} & \makecell{$123\inv=123$\\$312\inv=231$\\$213\inv=213$\\$132\inv=132$} & $\left(\begin{array}{ccc} 0&1&1 \\ 1&1&1 \\ 1&1&1 \end{array}\right)$ \\
        \makecell{$\sigma=(1,0,0,0)$;\\$n=4$;\\ $b=1$,\\$J(n,b)=8$}  & \makecell{(1,1,1,1)\\(3,0,0,1)\\(2,0,1,1)\\(1,2,0,1)\\(2,0,2,0)\\(4,0,0,0)\\(1,3,0,0)\\(1,1,2,0)}  & \makecell{(2,3,4,5)\\(4,2,3,5)\\(3,2,4,5)\\(2,4,3,5)\\(3,2,5,4)\\(5,2,3,4)\\(2,5,3,4)\\(2,3,5,4)} & \makecell{$1234\inv=1234$\\$3124\inv=2314$\\$2134\inv=2134$\\$1324\inv=1324$\\$2143\inv=2143$\\$4123\inv=2341$\\$1423\inv=1342$\\$1243\inv=1243$} & $\left(\begin{array}{cccc} 0&0&1&1 \\0&1&1&1 \\ 1&1&1&1 \\ 1&1&1&1 \end{array}\right)$ \\
        \makecell{$\sigma=(1,1)$;\\$n=2$;\\ $b=2$;\\$J(n,b)=2$}      & \makecell{(2,2)\\(3,1)}  & \makecell{(3,4)\\(4,3)} & \makecell{$12\inv=12$\\$21\inv=21$} & $\left(\begin{array}{cc} 1&1\\1&1 \end{array}\right)$  \\
        \makecell{$\sigma=(1,1,0)$;\\$n=3$;\\ $b=2$;\\$J(n,b)=6$}    & \makecell{(2,2,2)\\(2,3,1)\\(3,1,2)\\(3,3,0)\\(4,1,1)\\(4,2,0)} & \makecell{(3,4,5)\\(3,5,4)\\(4,3,5)\\(4,5,3)\\(5,3,4)\\(5,4,3)} & \makecell{$123\inv=123$\\$132\inv=132$\\$213\inv=213$\\$231\inv=312$\\$312\inv=231$\\$321\inv=321$} & $\left(\begin{array}{ccc} 1&1&1 \\ 1&1&1 \\ 1&1&1 \end{array}\right)$ \\\hline
    \end{tabular}
    \label{table:avQ-juggling}
    \caption{Sample values for juggling sequences and their images under $\theta$ for small $n$ and $b$.
    Note that we are using the permutation matrix notation system
    where rows are numbered from bottom to top in increasing order.}
\end{table}

The desired bijection can then be easily obtained by realizing that
the permutations mentioned in the proof of Theorem \ref{thm:juggling-enum}
are exactly the permutation inverses of those avoiding $Q_{b+2}$,
as is carefully fleshed out in the following theorem:

\begin{theorem}
    Let $\theta$ be a function from the set of ground state juggling sequences of period $n$ using $b$ balls to $Av_n(Q_{b+2})$
    given by $\theta((t_1,t_2,\dots,t_n))=\pi$
    where $\pi_{t_i+i-b}=i$ for all $i\in [n]$.
    Then $\theta$ is a bijection.
\end{theorem}
\begin{proof}
    Since $t_i$ is nonnegative for all $i\in [n]$, we have
    \[\pi_{t_i+i-b}=i
    \iff \pi_{i}=i+b-t_i
    \iff i \geq \pi_i-b
    = \pi_i - i - (b+2) + 2.
    \]
    Therefore, for a given $n,b$, the matrix $M$ defined in Chung and Graham's paper
    is exactly the transpose of the matrix that represents the avoidance of $Q_{b+2}$ as defined in Theorem \ref{thm:Q-matrix}.
    So the codomain of $\theta$ is indeed $Av_n(Q_{b+2})$.
    By Theorem \ref{thm:avQk},
    the number of ground state $b$-ball juggling sequences of period $n$
    is equal to
    the size of $Av_n(Q_{b+2})$
    since
    \[Av_n(Q_{b+2}) = \begin{cases}
        (b+1)!\times (b+1)^{n-b-1} = (b+1)^{n-b}b! &\text{if }n\geq b+2,\\
        n!&\text{otherwise.}
    \end{cases}\]

    Therefore, since $\theta$ is injective,
    it must also be bijective.
\end{proof}

\noindent We refer the reader to Table \ref{table:avQ-juggling} for sample values for small $n$ and $b$.


\subsection{Shrub forests of \texorpdfstring{$n$}{n} heaps}\label{sect:shrub}
\begin{definition}
    Let $\mathcal{P}_{3n}$ denote the set of permutations of length $3n$
    that avoid the patterns 231, 312 and 321
    and satisfies $\pi_{3i+1} < \pi_{3i+2}$ and $\pi_{3i+1} < \pi_{3i+3}$ for all $i\in [n-1]$.
\end{definition}

\begin{remark}
    $\mathcal{P}_{3n}$ is also known as the set of \textit{2-ary shrub forests of $n$ heaps} avoiding the patterns 231, 312 and 321.
\end{remark}

\bigskip
\noindent Gao and Kitaev \cite{gao-kitaev-2019} observed that $Av_n(Q_4)$ and the
$\mathcal{P}_{3n}$
are enumerated by the same OEIS sequence \href{https://oeis.org/A025192}{A025192}.
We will show a natural bijection between these two sets for all $n\geq 1$.

\begin{theorem} \label{thm:avQ-P}
    Let
    \[P_{3n}=\begin{cases}
        \{123,\, 132\} &\text{if }n=1,\\
        \{1 \oplus \tau \oplus \pi_{[2,3n-3]}\mid \tau\in \{123,\, 132,\, 213\}\text{ and } \pi\in \mathcal{P}_{3n-3}\} &\text{if }n\geq 2.
    \end{cases}.\]
    Then $P_{3n}=\mathcal{P}_{3n}$.
\end{theorem}

\begin{proof}
    It is easy to see that $\mathcal{P}_3=\{123,\, 132\}$.
    Let $\sigma:=1 \oplus \tau \oplus \pi_{[2,3n-3]}$
    for some $\tau\in \{123,\, 132,\, 213\}$
    and $\pi\in \mathcal{P}_{3n-3}$.
    It is clear that
    \[\abs{\sigma}=\abs{1 \oplus \tau \oplus \pi_{[2,3n]}}=1+3+\abs{\pi_{[2,3n]}}=4+3n-3-1=3n,\]
    so $\sigma$ is a $3n$-permutation.
    We know that $\pi\in \mathcal{P}_{3n-3}$,
    so $\sigma_{3i+1} < \sigma_{3i+2},\, \sigma_{3i+3}$ for $i=3,4,\dots,n-1$.
    By the definition of the direct sum $\oplus$
    we have $\sigma_1 < \sigma_{[2,4]} < \sigma_{[5,3n]}$,
    so $\sigma_{3i+1} < \sigma_{3i+2}$ and $\sigma_{3i+1} < \sigma_{3i+3}$ for $i=1$ and 2 as well.
    Since $\pi_{[2,3n-3]}$
    and $\tau=S_3\setminus \{231,\,312,\,321\}$
    both avoid 231, 312 and 321,
    so do the factors $\sigma_{[1,4]}$ and $\sigma_{[5,3n]}$.
    It is easy to see (by the definition of $\oplus$)
    that none of the forbidden patterns may span across $\sigma_{[1,4]}$ and $\sigma_{[5,3n]}$.
    Therefore $\sigma\in \mathcal{P}_{3n}$.\\

    We claim that $\abs{P_{3n}}=2\times 3^{n-1}$ for all $n\geq 1$.
    Indeed, $\abs{P_3}=2=2\times 3^0$.
    Suppose this is true for some $k\geq 2$.
    Since $\pi_1=1$ for all $\pi\in P_{3k}$,
    we must have that $\pi_{[2,3k]}$ is distinct for each $\pi\in P_{3k}$.
    Therefore $\abs{P_{3k+3}}=3\times \abs{P_{3k}}=2\times 3^{n-1}$,
    and the claim is true for all $n$.
    Since $P_{3n}\subseteq \mathcal{P}_{3n}$ and $\abs{P_{3n}}=\abs{\mathcal{P}}_{3n}$, they must be equal for all $n\geq 1$.
\end{proof}

\begin{theorem}\label{thm:avQ4-bij}
    Let $\theta: Av_n(Q_4)\rightarrow \mathcal{P}_{3n-3}$ be given by
    \begin{align*}
        \theta(12)&=123, \quad \theta(21)=132, \quad \text{and for }\abs{\pi}=n\geq 3,\\
        \theta(\pi) &=
        \begin{cases}
            1 \oplus 132 \oplus \theta\left(\pi_{[1,n-1]}\right)_{[2,3n]}, &\text{if }\pi_n=n,\\
            1 \oplus 123 \oplus \theta\left(\pi_{[1,n-2]}\,\pi_n\right )_{[2,3n]},&\text{if }\pi_{n-1}=n,\\
            1 \oplus 213 \oplus \theta\left(\pi_{[1,n-3]}\,\pi_{[n-1,n]} \right )_{[2,3n]} &\text{if }\pi_{n-2}=n.
        \end{cases}
    \end{align*}
    Then $\theta$ is a bijection.
\end{theorem}

\begin{figure}[!htbp]
    \centering
    \begin{tikzpicture}
	\begin{pgfonlayer}{nodelayer}
		\node [style=node] (0) at (0, 2) {};
		\node [style=node] (1) at (-1.5, 0) {};
		\node [style=node] (2) at (0, 0) {};
		\node [style=node] (3) at (1.5, 0) {};
		\node [style=none] (5) at (0, 3) {1};
		\node [style=none] (6) at (-1.5, -1) {2};
		\node [style=none] (7) at (0, -1) {3};
		\node [style=none] (8) at (1.5, -1) {4};
	\end{pgfonlayer}
	\begin{pgfonlayer}{edgelayer}
		\draw (0.center) to (1.center);
		\draw (0.center) to (2.center);
		\draw (0.center) to (3.center);
	\end{pgfonlayer}
\end{tikzpicture}
    \caption{The POP $Q_4$}
    \label{fig:Q4}
\end{figure}
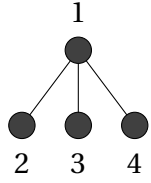

\begin{proof}
    Recall that $k$ can only be $\pi_{n}$, $\pi_{n-1}$ or $\pi_{n-2}$ by the proof of Theorem \ref{thm:Q-matrix}.
    So $\theta$ is defined on $Av_n(Q_4)$ for $n\geq 3$.
    Moreover, $\theta(Av_n(Q_4))\subseteq S_{3n-3}$ for all $n\geq 2$.
    The base case is: $\abs{\theta(12)}=\abs{\theta(21)}=3$.
    Suppose $\theta(Av_{k-1}(Q_4))\subseteq S_{3k-6}$ for some $k\geq 3$.
    Then if $\pi\in Av_k(Q_4)$,
    we have \[\abs{\theta(\pi)}=1+3+(3(k-1)-2+1)=3k=3(k+1)-3.\]
    So $\theta$ maps into $S_{3n-3}$ for all $n$.
    %
    Obviously $\theta(Av_2(Q_4))=\mathcal{P}_{3}$.
    Suppose $\theta(Av_{k-1}(Q_4))\subseteq \mathcal{P}_{3k-6}$ for some $k\geq 2$.
    Then by Theorem \ref{thm:avQ-P},
    $\theta$ indeed maps $Av_k(Q_4)$ into $\mathcal{P}_{3k-3}$.
    So $\theta(Av_{n}(Q_4))\subseteq \mathcal{P}_{3n-3}$ for all $n\geq 2$ by induction.\\

    Finally, it is clear from the definition of $\theta$ that it is injective.
    By Theorem \ref{thm:avQk}, $\abs{Av_n(Q_4)} = 3!\times 3^{n-3} = 2\times 3^{n-2}$ for $n\geq k$.
    Since $\abs{Av_n(Q_4)}=\abs{\mathcal{P}_{3n-3}}$ for all $n$,
    the map $\theta$ must be surjective as well.
\end{proof}

\begin{table}[!htbp]
    \centering
    \begin{tabular}{|c|c|c|}
        \hline
        $n$ & $Av_n(Q_4)$ & $\mathcal{P}_{3n-3}$ \\\hline
        2   & \makecell{12\\21} & \makecell{123\\132} \\
        3   & \makecell{123\\213\\132\\231\\312\\321}   & \makecell{124356\\124365\\123456\\123465\\132456\\132465} \\
        \hline
    \end{tabular}
    \label{}
    \caption{Sample values of $Av_n(Q_4)$ and their images under $\theta$ in $\mathcal{P}_{3n}$}
\end{table}

We note that our bijection is easily generalized:

\begin{definition}
    Let $Q_{k,j}$ be the POP of size $k$ where $i<j$ for all $i,j\in [k]$ where $i\neq j$,
    as illustrated in Figure \ref{fig:Qkj}.
    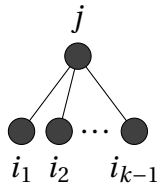
\begin{figure}[!htbp]
        \centering
        \begin{tikzpicture}
	\begin{pgfonlayer}{nodelayer}
		\node [style=node] (0) at (0, 2) {};
		\node [style=node] (1) at (-1.5, 0) {};
		\node [style=node] (2) at (-0.5, 0) {};
		\node [style=node] (3) at (1.5, 0) {};
		\node [style=none] (4) at (0.5, 0) {$\dots$};
		\node [style=none] (5) at (0, 3) {$j$};
		\node [style=none] (6) at (-1.5, -1) {$i_1$};
		\node [style=none] (7) at (-0.5, -1) {$i_2$};
		\node [style=none] (8) at (1.5, -1) {$i_{k-1}$};
	\end{pgfonlayer}
	\begin{pgfonlayer}{edgelayer}
		\draw (0.center) to (1.center);
		\draw (0.center) to (2.center);
		\draw (0.center) to (3.center);
	\end{pgfonlayer}
\end{tikzpicture}
        \caption{The POP $Q_{k,j}$, where $\{j,i_1,i_2,\dots,i_{k-1}\}=[k]$}
        \label{fig:Qkj}
    \end{figure}
\end{definition}

\begin{theorem}
    There is a natural bijection between $Av_n(Q_{4,j})$ and $\mathcal{P}_{3n-3}$ for all $1\leq j\leq 4$.
\end{theorem}

\begin{proof}
    Due to the symmetry of $Q_{k,j}$ and the fact that $Q_k=Q_{k,1}$,
    it is not hard to see that $\pi\in Av_n(Q_{k})$
    if and only if $\pi_j\, \pi_{[2,j-1]} \,\pi_1 \,\pi_{[j+1,n]}\in Av_n(Q_{k,j})$.
    We can then obtain a bijection from between $Av_n(Q_{4,j})$ and $\mathcal{P}_{3n-3}$
    for all $1\leq j\leq 4$
    by composing $\theta$ from Theorem \ref{thm:avQ4-bij}
    with the bijection $\pi \mapsto \pi_j \,\pi_{[2,j-1]}\, \pi_1 \,\pi_{[j+1,n]}$.
\end{proof}

\section{Levels in compositions of \texorpdfstring{$n$}{n} of ones and twos}\label{chap:levels}

\begin{definition}
    Define $P_k$ to be a POP with $k$ elements where $1>3$. Its Hasse diagram is illustrated below.
    \begin{figure}[!htbp]
        \centering
        \begin{tikzpicture}
	\begin{pgfonlayer}{nodelayer}
		\node [style=node] (0) at (0, 1.5) {};
		\node [style=node] (1) at (0, 0) {};
		\node [style=node] (2) at (1.25, 0) {};
		\node [style=node] (3) at (2.5, 0) {};
		\node [style=none] (5) at (4.25, 0) {\dots};
		\node [style=none] (6) at (0, -1) {3};
		\node [style=none] (7) at (1.25, -1) {4};
		\node [style=none] (8) at (2.5, -1) {5};
		\node [style=none] (9) at (5.75, -1) {$k$};
		\node [style=none] (10) at (0, 2.5) {1};
		\node [style=node] (11) at (5.75, 0) {};
		\node [style=none] (12) at (-1, -1) {2};
		\node [style=node] (13) at (-1, 0) {};
	\end{pgfonlayer}
	\begin{pgfonlayer}{edgelayer}
		\draw [in=90, out=-90] (0) to (1);
	\end{pgfonlayer}
\end{tikzpicture}
        \caption{The POP $P_k$}
        \label{fig:Pk}
    \end{figure}
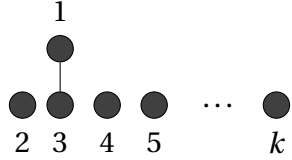
\end{definition}

\begin{definition}
    A \textit{composition of $n$} is a way of writing $n$ as the sum of positive integers that sum to $n$,
    that is, an expression $i_1+i_2+\cdots + i_k=n$ for some $k\geq 1$.
    A \textit{composition of $n$ of ones and twos} has the added restriction that $i_j\in \{1,2\}$ for all $j\in [n]$.
    In this paper, we will refer to a composition of $n$ of ones and twos simply as a ``composition of $n$'',
    and call the set $\mathcal{C}_n$.
\end{definition}

\begin{definition}
    A \textit{level} in a composition of $n$ (of ones and twos) is a pair of consecutive ones or twos separated by a $+$ sign.
    We define a \textit{marked composition} of $n$ to be a composition of $n$
    with exactly one level marked with a line above the pair of ones or twos.
    We may denote a single summand that is part of a level by including a line above it,
    i.e. $\overline{1}+\overline{1}=\overline{1+1}$ and $\overline{2}+\overline{2}=\overline{2+2}$.
    We denote the set of marked compositions of $n$ by $\mathcal{L}_n$.
\end{definition}

\noindent We refer the reader to Table \ref{table:av9-comps-levels} for examples.\\

Gao and Kitaev \cite{gao-kitaev-2019} discovered that the sequence $\abs{Av_n(P_4)}_{n\geq 1}$ corresponds to the OEIS sequence \href{http://oeis.org/A045925}{A045925}
that enumerates the number of levels in all compositions of $n+1$ of ones and twos.
In this section, we will demonstrate an explicit bijection between the two sets.
To do so, we first construct an explicit bijection from $Av_n(P_3)$ to $\mathcal{C}_n$.

\begin{table}[!htbp]
    \centering
    \begin{tabular}{
        >{\columncolor[HTML]{FFF2CC}}c |
        >{\columncolor[HTML]{DAE8FC}}c |
        >{\columncolor[HTML]{DAE8FC}}c |
        >{\columncolor[HTML]{E2EFDA}}c |
        >{\columncolor[HTML]{E2EFDA}}c
        }
        $n$	&$\abs{\mathcal{C}_n}$ & $\mathcal{C}_n$                              & $\abs{\mathcal{L}_n}$ & $\mathcal{L}_n$   \\\hline
        1	&   1   &   $1$	                                                      &           0           &    none    \\\hline
        2	&   2   &   $1+1, 2$                                                  &           1           &    $\overline{1+1}$     \\\hline
        3	&   3   &   \parbox{50pt}{\makecell{$1+1+1$,\\$2+1$,\\$1+2$}} 	      &           2           &    \parbox{50pt}{\makecell{$\overline{1+1}+1$,\\$1+\overline{1+1}$}}     \\\hline
        4   &   5   &   \makecell{$1+1+1+1$,\\$1+1+2,$\\$1+2+1$,\\$2+1+1$\\$2+2$} &           6           &    \makecell{$\overline{1+1}+1+1$,\\$1+\overline{1+1}+1$,\\$1+1+\overline{1+1}$,\\$\overline{1+1}+2$,\\$2+\overline{1+1}$,\\$\overline{2+2}$}     \\
    \end{tabular}
    \caption{Compositions and levels of $n$ for small $n$}
    \label{table:av9-comps-levels}
\end{table}

\subsection{Compositions of \texorpdfstring{$n$}{n}}\label{sect:av9-comps}

\begin{lemma}\label{lemma:av9-comps}
    Let $F(n)$ denote the $n$th Fibonacci number, where $F(0)=0$ and $F(1)=1$.
    The size of $\mathcal{C}_n$ is $F(n)$ for all $n\geq 1$.
\end{lemma}
\begin{proof}
    The size of $\mathcal{C}_n$ for $n=1$ and 2 can be verified easily.
    For $n\geq 3$, observe that
    \[c\in \mathcal{C}_{n-1} \iff c+1\in \mathcal{C}_{n}\quad \text{and}\quad c\in \mathcal{C}_{n-2} \iff c+2\in \mathcal{C}_{n},\]
    so $\abs{\mathcal{C}_n}=\abs{\mathcal{C}_{n-1}}+\abs{\mathcal{C}_{n-2}}$, and $\abs{\mathcal{C}_n}=F_n$.
\end{proof}

\begin{lemma}\label{lemma:av9-avn(P3)}
    The size of $Av_n(P_3)$ is $F(n)$ for all $n\geq 1$.
    Moreover, all $n$-permutations avoiding $P_3$ for $n\geq 2$ are sum decomposable.
\end{lemma}
\begin{proof}
    It is easy to see that $\abs{Av_1(P_3)}=1=F(1)$ and $\abs{Av_2(P_3)}=2=F(2)$.
    We refer the reader to Table \ref{table:av9-avn(P3)-avn(P4)} for examples.
    It is not difficult to see that
    for $n\geq 3$,
    \[\sigma\in Av_{n-1}(P_3) \iff 12[\sigma,1] \in Av_n(P_3)
    \quad\text{and}\quad
    \sigma\in Av_{n-2}(P_3) \iff 12[\sigma,21] \in Av_n(P_3).
    \]
    So by Theorem \ref{thm:12,21},
    \[\abs{Av_n(P_3)}=\abs{Av_{n-1}(P_3)}+\abs{Av_{n-2}(P_3)}=F(n)+F(n-1)=F(n+1)\]
    for all $n\geq 3$ as well, proving the statements.
\end{proof}

\bigskip
The previous two lemmas suggest that there is a natural bijection from
the set $\mathcal{C}_n$ to $Av_n(P_3)$.
Indeed there is, as we will see in the following theorem:

\begin{table}[!htbp]
    \centering
    \begin{tabular}{
        >{\columncolor[HTML]{FFF2CC}}c |
        >{\columncolor[HTML]{FFF2CC}}c |
        >{\columncolor[HTML]{DAE8FC}}c |
        >{\columncolor[HTML]{DAE8FC}}c
        }
        $n$	& $F(n)$&   \thead{Compositions of $n$}  &  \thead{Permutations in $Av_n(P_3)$}  \\\hline
        1	&   1	&       1       &     1                     \\\hline
        2	&   2	&     $1+1$     &     $12=12[1,1]$          \\
        	&   	&       2       &     $21=21[1,1]$          \\\hline
        3	&   3	&    $1+1+1$    &     $123=123[1,1,1]$      \\
            &   	&     $1+2$     &     $132=12[1,21]$        \\
            &       &     $2+1$     &     $213=12[21,1]$        \\\hline
        4   &   5   &   $1+1+1+1$   &     $1234=1234[1,1,1,1]$  \\
            &       &    $2+1+1$    &     $2134=123[21,1,1]$    \\
            &       &    $1+1+2$    &     $1243=123[1,1,21]$    \\
            &       &    $1+2+1$    &     $1324=123[1,21,1]$    \\
            &       &     $2+2$     &     $2143=12[21,21]$      \\
    \end{tabular}
    \caption{Compositions of $n$ and their images under $f$ for small $n$}
    \label{table:av9-f-correspondence}
\end{table}

\begin{theorem}\label{thm:av9-comps,P3}
    Let $f:\mathcal{C}_n\rightarrow Av_n(P_3)$
    defined as
    \[f(r_1+r_2+\cdots+r_k)=123\cdots k[\alpha_1,\alpha_2,\dots,\alpha_k]\]
    where $r_1+r_2+\cdots+r_k$ is a composition of $n$ of ones and twos,
    and
    \[\alpha_i=\begin{cases}
        1&\text{if }r_i=1\\
        21&\text{if }r_i=2
    \end{cases}.\]
    Then $f$ is a bijection.
\end{theorem}
\begin{proof}
    We refer the reader to Table \ref{table:av9-f-correspondence} for examples for small $n$.
    First we check that if $r_1+r_2+\cdots+r_k$ is a composition of $n$,
    then $f(r_1+r_2+\cdots+r_k)$ is a permutation in $S_n$:
    \begin{align*}
        \abs{f(r_1+r_2+\cdots r_k)}
        & = \abs{123\cdots k[\alpha_1,\alpha_2,\dots,\alpha_k]} \\
        & = \abs{\alpha_1}+\abs{\alpha_2}+\cdots \abs{\alpha_k} \\
        & = r_1+r_2+\cdots r_k \\
        & = n.
    \end{align*}

    Moreover, it is clear that
    if $\alpha_i\in \{1,21\}$
    for all $i\in [k]$,
    then the permutation \[\pi=123\cdots k[\alpha_1,\alpha_2,\dots,\alpha_k]\] avoids $P_3$.
    It is easy to see that $f$ is injective by definition.
    By Lemmas \ref{lemma:av9-comps} and \ref{lemma:av9-avn(P3)},
    $\mathcal{C}_n$ and $Av_n(P_3)$ both have $F(n)$ elements,
    so $f$ maps bijectively onto $Av_n(P_3)$.


\end{proof}

\subsection{Levels of compositions of \texorpdfstring{$n$}{n}}\label{sect:av9-levels}

\begin{table}[!htbp]
    \centering
    \begin{tabular}{
        >{\columncolor[HTML]{FFF2CC}}c |
        >{\columncolor[HTML]{DAE8FC}}c |
        >{\columncolor[HTML]{DAE8FC}}c |
        >{\columncolor[HTML]{E2EFDA}}c |
        >{\columncolor[HTML]{E2EFDA}}c
        }
        $n$	&$\abs{Av_n(P_3)}$  & $Av_n(P_3)$                                   &   $\abs{Av_n(P_4)}$  &    $Av_n(P_4)$   \\\hline
        1	&   1               &   1	                                        &      1               &    1    \\\hline
        2	&   2               &   12, 21                                      &      2               &    12, 21     \\\hline
        3	&   3               &   123, 132, 213 	                            &      6               &    \makecell{123,132,213,\\231,312,321\\\phantom{xxxxxxxxxxxxxxxxxxxx}}     \\\hline
        4   &   5               &   \makecell{1234, 1243, 1324,\\2134, 2143}    &      12              &    \makecell{1234,1243,1324,1342,\\1423,1432,2134,2143,\\2341,2431,3142,3241}     \\
    \end{tabular}
    \caption{$Av_n(P_3)$ and $Av_n(P_4)$ for small $n$}
    \label{table:av9-avn(P3)-avn(P4)}
\end{table}

\begin{lemma}\label{lemma:av9-levels}
    For $n\geq 1$, the size of the set $\mathcal{L}_{n}$ is $(n-1)F(n-1)$.
    Moreover, for $n\geq 3$, we can partition $\mathcal{L}_{n}$ into four sets:
    \begin{align*}
        A'_n &= \{\text{marked compositions of } n \text{ that end with } 1 \text{ (not }\overline{1})\}\\
        B'_n &= \{\text{marked compositions of } n \text{ that end with } 2 \text{ (not }\overline{2})\}\\
        C'_n &= \{\text{marked compositions of } n \text{ that end with } \overline{1+1}\}\\
        D'_n &= \{\text{marked compositions of } n \text{ that end with } \overline{2+2}\}
    \end{align*}
\end{lemma}

\begin{proof}
    For $n\in [3]$, it is easy to check that $\abs{L_{n}}=(n)-1!=(n-1)F(n-1)$,
    as we show in Table \ref{table:av9-comps-levels}.
    For $n\geq 4$, we proceed by induction.
    Suppose that for some $k\geq 4$, the statement is true for all $n\in [k-1]$.
    It is clear that the union of $A'_k$, $B'_k$, $C'_k$, and $D'_k$ together make up $\mathcal{L}_{k}$,
    and that the following statements are true for all $n$:
    \begin{align*}
        m\in \mathcal{L}_n &\iff m+1\in \mathcal{L}_{n+1},\qquad
        &c\in \mathcal{C}_{n-1} &\iff c+\overline{1+1}\in \mathcal{L}_{n+1}\\
        m\in \mathcal{L}_{n-1} &\iff m+2\in \mathcal{L}_{n+1},\qquad
        &c\in \mathcal{C}_{n-3} &\iff c+\overline{2+2}\in \mathcal{L}_{n+1}.
    \end{align*}

    \noindent So by the inductive hypothesis,
    \[\abs{A'_k}=(k-2)F(k-2)
    \qquad \text{and} \qquad
    \abs{B'_k}=(k-3)F(k-3),\]

    \noindent and by Lemma \ref{lemma:av9-comps},
    \[\abs{C'_k}=F(k)
    \qquad \text{and} \qquad
    \abs{D'_k}=F(k-2)\]

    \noindent Therefore the size of $\mathcal{L}_k$ is
    \begin{align*}
        &\abs{A'_k}+\abs{B'_k}+\abs{C'_k}+\abs{D'_k}\\
        =&\quad (k-1)F(k-1)+(k-2)F(k-2)+F(k)+F(k-2)\\
        =&\quad kF(k-1)-F(k-1)+kF(k-2)-2F(k-2)+F(k)+F(k-2)\\
        =&\quad k(F(k-1)+F(k-2))-F(k-1)-F(k-2)+F(k)\\
        =&\quad kF(k),
    \end{align*}

    \noindent and the statement is true by induction on $n$.
\end{proof}

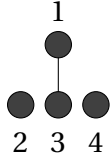
\begin{figure}[!htbp]
    \centering
    \begin{tikzpicture}
	\begin{pgfonlayer}{nodelayer}
		\node [style=node] (0) at (0, 1.6) {};
		\node [style=node] (1) at (0, 0) {};
		\node [style=node] (2) at (-1, 0) {};
		\node [style=node] (3) at (1, 0) {};
		\node [style=none] (4) at (0, 2.5) {1};
		\node [style=none] (5) at (0, -1) {3};
		\node [style=none] (6) at (-1, -1) {2};
		\node [style=none] (7) at (1, -1) {4};
	\end{pgfonlayer}
	\begin{pgfonlayer}{edgelayer}
		\draw (0.center) to (1.center);
	\end{pgfonlayer}
\end{tikzpicture}
    \caption{The POP $P_4$}
    \label{fig:Av9-P4}
\end{figure}

\begin{lemma}\label{lemma:av9-avn(P4)}
    For $n\geq 1$, the size of $Av_n(P_4)$ is $nF(n)$.
    Moreover, for $n\geq 4$,  we can partition $Av_n(P_4)$ into four sets.
    Specifically, $Av_n(P_4)=A_n\sqcup B_n\sqcup C_n\sqcup D_n$ where
    \begin{align*}
        A_n:&=\{12[1,\sigma]\,:\,\sigma\in Av_{n-1}(P_4)\},\\
        B_n:&=\{12[21,\sigma]\,:\,\sigma\in Av_{n-2}(P_4)\},\\
        C_n:&=\{21[\sigma,1]\,:\,\sigma\in Av_{n-1}(P_3)\}, \quad \text{and}\\
        D_n:&=\{3142[1,1,\sigma,1]\,:\,\sigma\in Av_{n-3}(P_3)\}.
    \end{align*}

\end{lemma}
\begin{proof}
    For $n\leq 3$, it is easy to check that $\abs{Av_n(P_4)}=n!=nF(n)$.
    We refer the reader to Table \ref{table:av9-avn(P3)-avn(P4)} for examples.\\

    For $n\geq 4$, we may proceed by induction.
    First, we leave it to the reader to check that the following four claims are true for all $n\geq 4$:
    \begin{enumerate}[1.]
        \item $12[1,\sigma] \in Av_{n}(P_4) \iff \sigma\in Av_{n-1}(P_3)$,
        \item $12[21,\sigma]  \in Av_{n}(P_4)\iff \sigma\in Av_{n-2}(P_3)$,
        \item $21[\sigma,1]\in Av_n(P_4) \iff \sigma\in Av_{n-1}(P_3)$, \quad and
        \item $3142[1,1,\sigma,1]\in Av_n(P_4) \iff \sigma\in Av_{n-3}(P_4)$.
    \end{enumerate}
    These imply that $A_n,B_n,C_n$ and $D_n$ are subsets of $Av_n(P_4)$.\\

    Now suppose that for some $k\geq 4$,
    the statement is true for all $n\in[k-1]$.
    Then by our inductive hypothesis,
    we have that $A_k$ contains $(k-1)F(k-1)$ elements
    and $B_k$ contains $(k-2)F(k-2)$ elements.\\

    By Lemma \ref{lemma:av9-avn(P3)},
    $Av_n(P_3)$ is counted by the $(n+1)$th Fibonacci number for all $n$,
    so $C_k$ contains $F(k)$ elements and
    $D_k$ contains $F(k-2)$ elements.\\

    It is not hard to see that the four sets are disjoint,
    so the total number of items in the four sets is
    \begin{align*}
        &\abs{A_k}+\abs{B_k}+\abs{C_k}+\abs{D_k}\\
        =&\quad (k-1)F(k-1)+(k-2)F(k-2)+F(k)+F(k-2)\\
        =&\quad kF(k-1)-F(k-1)+kF(k-2)-2F(k-2)+F(k)+F(k-2)\\
        =&\quad k(F(k-1)+F(k-2))-F(k-1)-F(k-2)+F(k)\\
        =&\quad kF(k).
    \end{align*}

    \noindent It remains to check
    that any $k$-permutation avoiding $P_4$ indeed lies in $A_k, B_k, C_k$ or $D_k$.
    By Theorem \ref{thm:12,21},
    it suffices to show that
    \begin{enumerate}[(a)]
        \item 12, 21 and 3142 are the only simple permutations that avoid $P_4$,
        \item if $21[\alpha,\beta]$ avoids $P_4$ and $\alpha\neq 1$ is skew-sum indecomposable, then $\beta=1$, and
        \item if $12[\alpha,\beta]$ avoids $P_4$ and $\alpha$ is sum indecomposable, then $\alpha=1$ or 21.
    \end{enumerate}

    For (a), it is clear that 12 and 21 avoid $P_4$.
    Since 2413 contains $P_4$, any simple permutation of length at least 4 avoiding $P_4$
    must contain $3142$ by Theorem \ref{thm:separable}.
    Let $\pi$ be such a simple permutation.
    We can view it as a lattice matrix.
    This is shown in Figure \ref{fig:Av9-P4},
    with some alterations explained in the caption.
    The blocks $\alpha_{13},\, \alpha_{14},\, \alpha_{23}$ and $\alpha_{24}$ are adjacent to the bullet point representing the number 4,
    so they must be trivial since $\pi$ is simple.
    Finally, $\alpha_{51}$ must be trivial,
    otherwise $\pi$ would be sum decomposable.\\

    For (b), it is clear that $\beta$ can have size at most 1,
    since otherwise $21[\alpha,\beta]$ would contain $P_4$ for any $\abs{\alpha}\geq 2$.\\

    For (c), suppose $\alpha$ is sum indecomposable and of length at least 3.
    Since $\abs{\beta}\geq 1$, $\alpha$ must avoid $P_3$.
    By (a), $\alpha$ is an inflation of $21$ or $3142$.
    The only inflation of 21 avoiding $P_3$ is itself,
    while 3142 contains $P_3$. So $\alpha=1$ or 21.

    \begin{figure}
        \begin{center}
            \begin{tabular}{ccccccccc}
                1             &  \big |   &      2        & \big |    & $\alpha_{13}$ & \big |    & $\alpha_{14}$ & \big |    & 4      \\
                ------        &  $\Plus$  &     ------    &  $\Plus$  &     ------    & $\Bullet$ &     ------    &  $\Plus$  & ------ \\
                1             &  \big |   &      2        & \big |    & $\alpha_{23}$ & \big |    & $\alpha_{24}$ & \big |    & 4      \\
                ------        & $\Bullet$ &     ------    &  $\Plus$  &     ------    &  $\Plus$  &     ------    &  $\Plus$  & ------ \\
                1             &  \big |   &      2        & \big |    &        3      & \big |    &        3      & \big |    & 4      \\
                ------        &  $\Plus$  &     ------    &  $\Plus$  &     ------    &  $\Plus$  &     ------    & $\Bullet$ & ------ \\
                1             &  \big |   &      2        & \big |    &        3      & \big |    &        3      & \big |    & 4      \\
                ------        &  $\Plus$  &     ------    & $\Bullet$ &     ------    &  $\Plus$  &     ------    &  $\Plus$  & ------ \\
                $\alpha_{51}$ &  \big |   &      2        & \big |    &        3      & \big |    &        3      & \big |    & 4
            \end{tabular}
            \caption{The lattice matrix $L_{3142}(\pi)$, with alterations.
            The 1s corresponding to the pattern $3142$ are replaced by bullet points.
            For all $i,\, j\in [5]$, $\alpha_{ij}$ is replaced by some $\ell\in [4]$
            if and only if $\pi$ would contain $P_4$ if $\alpha_{ij}$ were non-trivial
            and any point $\alpha_{ij}$ could take the place of the point labelled $\ell$ in the POP.}
            \label{fig:Av9(P4)}
        \end{center}
    \end{figure}
\end{proof}

Once again, the previous two lemmas suggest that there is a natural bijection from
the set $\mathcal{L}_{n+1}$ to $Av_n(P_4)$ - and indeed there is,
as we will see in the main theorem of this section:\\

\begin{theorem}
    Let $g:\mathcal{L}_{n+1} \rightarrow Av_n(P_4)$, where $n\geq 0$ and $r_i\in \{1,2,\overline{1},\overline{2}\}$ for $i\in [k]$ and $k\geq 1$
    \begin{align*}
        g(r_1+r_2+\cdots r_k)&=\begin{cases}
            1&\text{ if }\quad k=1=r_1,\\
            21&\text{ if }\quad k=2,\quad r_{1}=r_2=\overline{1},\\
            12[1,g(r_1+\cdots +r_{k-1})]&\text{ if }\quad r_k=1,\\
            12[21,g(r_1+\cdots +r_{k-1})]&\text{ if }\quad r_k=2,\\
            21[f(r_1+\cdots +r_{k-2}),1]&\text{ if }\quad r_{k-1}=r_k=\overline{1},\\
            3142[1,1,f(r_1+\cdots +r_{k-2}),1]&\text{ if }\quad r_{k-1}=r_k=\overline{2}.
        \end{cases}
    \end{align*}
    Then $g$ is a bijection.
\end{theorem}

\begin{table}[!htbp]
    \centering
    \begin{tabular}{
        >{\columncolor[HTML]{FFF2CC}}c |
        >{\columncolor[HTML]{E2EFDA}}c |
        >{\columncolor[HTML]{E2EFDA}}c |
        >{\columncolor[HTML]{FCE4D6}}c |
        >{\columncolor[HTML]{FCE4D6}}c
        }
        $n$	& \parbox{70pt}{\thead{Marked\\compositions\\of $n+1$}}  &  \parbox{80pt}{\thead{Sum\\decomposables\\in $Av_n(P_4)$}}  & \parbox{80pt}{\thead{Marked\\compositions\\of $n+1$}}  &  \parbox{90pt}{\thead{Sum\\indecomposables\\in $Av_n(P_4)$}}    \\\hline
        1   &          -         		  &          -          &   $\overline{1+1}$         &   $1$                  \\\hline
        2   &   $\overline{1+1}+1$		  & $12=12[1,1]$        &   $1+\overline{1+1}$       &   $21=21[1,1]$         \\\hline
        3   &   $\overline{1+1}+1+1$,	  & $123=12[1,12]$      &   $1+1+\overline{1+1}$,	 &   $231=21[12,1]$	      \\
            &   $1+\overline{1+1}+1$	  & $132=12[1,21]$      &   $2+\overline{1+1}$	     &   $321=21[1,21]$	      \\
            &   $\overline{1+1}+2$	      & $213=12[21,1]$      &   $\overline{2+2}$	     &   $312=21[1,12] $      \\\hline
        4   &   $\overline{1+1}+1+1+1$	  & $1234=12[1,123]$    &   $1+1+1+\overline{1+1}$	 &   $2341=21[123,1]$     \\
            &   $1+\overline{1+1}+1+1$	  & $1243=12[1,132]$    &   $2+1+\overline{1+1}$	 &   $2431=21[132,1]$     \\
            &   $1+1+\overline{1+1}+1$	  & $1324=12[1,213]$    &   $1+2+\overline{1+1}$	 &   $3241=21[213,1]$     \\
            &   $\overline{1+1}+2+1$	  & $1342=12[1,231]$    &   $1+\overline{2+2}$	     &   $3142=3142[1,1,1,1]$ \\
            &   $2+\overline{1+1}+1$	  & $1432=12[1,321]$    &                            &                        \\
            &   $\overline{2+2}+1$	      & $1423=12[1,312]$    &                            &                        \\
            &   $\overline{1+1}+1+2$	  & $2134=12[21,12]$    &                            &                        \\
            &   $1+\overline{1+1}+2$	  & $2143=12[21,21]$    &                            &                        \\
    \end{tabular}
    \caption{Marked compositions of $n+1$ and their images under $g$ for small $n$}
    \label{table:av9-g-correspondence}
\end{table}

\begin{proof}
    Recall from Theorem \ref{thm:av9-comps,P3} that $f$ is a bijection,
    so it is clear from the definition that $g$ is injective.
    We refer the reader to Tables \ref{table:av9-seqs-levels}, \ref{table:av9-seqs-P4} and
    \ref{table:av9-g-correspondence} for enumerations for small $n$.\\

    By Lemmas \ref{lemma:av9-levels} and \ref{lemma:av9-avn(P4)},
    the sets $\mathcal{L}_{n+1}$ and $Av_n(P_4)$ both contain $nF(n)$ elements,
    so $g$ must be bijective. \\

    \noindent In addition, it can easily be seen that the inverse of $g$ is the following:
    \[g\inv(\pi)=\begin{cases}
        g\inv(\alpha_2) + \abs{\alpha_1} &\text{ if }\quad \pi=12[\alpha_1,\alpha_2],\\
        f\inv(\alpha) + \overline{1+1} &\text{ if }\quad \pi=21[\alpha,1],\\
        f\inv(\alpha) + \overline{2+2} &\text{ if }\quad \pi=3142[1,1,\alpha,1].\\
    \end{cases}\]
\end{proof}

\begin{table}[!htbp]
    \centering
    \begin{tabular}{
        >{\columncolor[HTML]{FFF2CC}}c |
        >{\columncolor[HTML]{FFF2CC}}c
        >{\columncolor[HTML]{E2EFDA}}c
        >{\columncolor[HTML]{FCE4D6}}c }
        $n$& \thead{$\abs{\mathcal{L}_{n+1}}$} & \parbox{160pt}{\thead{Number of compositions\\in $\mathcal{L}_{n+1}$ where the\\last summand is not marked}} & \parbox{160pt}{\thead{Number of compositions\\in $\mathcal{L}_{n+1}$ where the\\last two terms are marked}} \\\hline
        1  & 1   & 0   & 1   \\
        2  & 2   & 1   & 1   \\
        3  & 6   & 3   & 3   \\
        4  & 12  & 8   & 4   \\
        5  & 25  & 18  & 7   \\
        6  & 48  & 37  & 11  \\
        7  & 91  & 73  & 18  \\
        8  & 168 & 139 & 29  \\
        9  & 306 & 259 & 47  \\
        10 & 550 & 474 & 76  \\
        11 & 979 & 856 & 123 \\
        $k$ & $kF(k)$ & $(k-1)F(k-1)+(k-2)F(k-2)$ & $L(k-1)=F(k-2)+F(k)$ \\
    \end{tabular}
    \caption{Enumerating marked compositions for small $n$}
    \label{table:av9-seqs-levels}
    \end{table}

\begin{table}[!htbp]
    \centering
    \begin{tabular}{
        >{\columncolor[HTML]{FFF2CC}}c |
        >{\columncolor[HTML]{FFF2CC}}c
        >{\columncolor[HTML]{E2EFDA}}c
        >{\columncolor[HTML]{FCE4D6}}c }
        $n$& $\abs{Av_n(P_4)}$  & \parbox{160pt}{\thead{Number of\\sum decomposable\\$n$-permutations avoiding $P_4$}}  & \parbox{160pt}{\thead{Number of\\sum indecomposable\\$n$-permutations avoiding $P_4$}}   \\\hline
        1  & 1   & 0   & 1   \\
        2  & 2   & 1   & 1   \\
        3  & 6   & 3   & 3   \\
        4  & 12  & 8   & 4   \\
        5  & 25  & 18  & 7   \\
        6  & 48  & 37  & 11  \\
        7  & 91  & 73  & 18  \\
        8  & 168 & 139 & 29  \\
        9  & 306 & 259 & 47  \\
        10 & 550 & 474 & 76  \\
        11 & 979 & 856 & 123 \\
        $k$ & $kF(k)$ & $(k-1)F(k-1)+(k-2)F(k-2)$ & $L(k-1)=F(k-2)+F(k)$ \\
    \end{tabular}
    \caption{Enumerating sum decomposable and sum indecomposable permutations for small $n$}
    \label{table:av9-seqs-P4}
    \end{table}

\section{Partial sums of signed displacements}\label{chap:av10}

\begin{definition}
    Let $R_k$ be the POP of size $k$ where $1>k$.

    \begin{figure}[!htbp]
        \centering
        \begin{tikzpicture}
	\begin{pgfonlayer}{nodelayer}
		\node [style=node] (0) at (0, 1.5) {};
		\node [style=node] (1) at (0, 0) {};
		\node [style=node] (2) at (1.25, 0) {};
		\node [style=node] (3) at (2.5, 0) {};
		\node [style=none] (5) at (4.25, 0) {\dots};
		\node [style=none] (6) at (0, -1) {$k$};
		\node [style=none] (7) at (1.25, -1) {2};
		\node [style=none] (8) at (2.5, -1) {3};
		\node [style=none] (9) at (5.75, -1) {$k-1$};
		\node [style=none] (10) at (0, 2.5) {1};
		\node [style=node] (11) at (5.75, 0) {};
	\end{pgfonlayer}
	\begin{pgfonlayer}{edgelayer}
		\draw [in=90, out=-90] (0) to (1);
	\end{pgfonlayer}
\end{tikzpicture}
        \caption{The POP $R_k$}
        \label{fig:Rk}
    \end{figure}
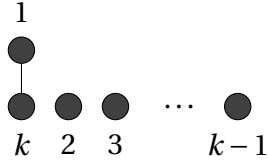
\end{definition}

It is easy to see that the avoidance set of $R_2$ contains only the identity permutations.
Observe that $R_3$ is equivalent to the POP $P_3$ introduced in Section \ref{chap:levels},
where we showed that its avoidance set is enumerated by the well-known Fibonacci numbers.
Gao and Kitaev \cite{gao-kitaev-2019} enumerated the avoidance sets of $R_4$ and $R_5$ in Theorems 14 and 33 of their paper respectively.
Moreover, they showed that for all $k\geq 3$, the avoidance set of $R_k$ is in one-to-one correspondence with $n$-permutations
such that for each cycle $c$,
the smallest integer interval containing all elements of $c$ has at most $k-1$ elements.\\

They also observed that $Av_n(R_4)$ and the set of $n$-permutations for which the partial sums of signed displacements do not exceed 2 (to be defined later)
are the same size for all $n\geq 1$, and asked if there is an interesting bijection on one set to the other.
We construct such a bijection in this section by analyzing the two sets in detail.

\subsection{Analysing \texorpdfstring{$Av_n(R_4)$}{Av-n(R-4)}}

\begin{theorem}\label{thm:av10-an}
    If $\pi$ is an $n$-permutation that avoids $R_4$
    and $n\geq 5$,
    then $\pi$ is sum decomposable.
    Moreover, if $12[\alpha,\beta]$ avoids $R_4$ for permutations $\alpha$ and $\beta$
    where $\alpha$ is sum indecomposable,
    then $\alpha$ is 1, 21, 231, 321, 312 or 2413.
\end{theorem}
\begin{proof}
    We claim that the only simple permutations avoiding $R_4$ are 1, 12, 21 and 2413.
    Recall that any simple $n$-permutation contains the patterns 132, 213 and 312 if $n\geq 4$.
    Thus, we can write such a simple permutation as the string of factors
    \[\alpha^{(1)} \, p \,\alpha^{(2)} \, q \,\alpha^{(3)} \, r\,\alpha^{(4)}\]
    where $\alpha_i$ are possibly empty factors (of length 0)
    and $p$, $q$ and $r$ are factors of length 1
    where $\text{red}(p\, q\, r)=312$.
    It is clear that if $\alpha^{(2)}$ and $\alpha^{(3)}$ were not empty,
    if $\alpha^{(1)}>p$ or if $\alpha^{(4)}<p$ then the permutation would contain $R_4$.
    So $\alpha^{(2)}$ and $\alpha^{(3)}$ are empty and
    if $\alpha^{(1)}$ and $\alpha^{(4)}$ are not empty then we must have $\alpha^{(1)}<p<\alpha^{(4)}$.
    However, this implies that if $\alpha_4$ is not empty then the permutation must be sum decomposable, so $\alpha^{(4)}$ is empty.
    If $\abs{\alpha^{(1)}}=t\geq 2$, then we must have $\alpha^{(2)}_{[1,t-1]}<q$ otherwise the permutation contains $R_4$.
    This forces $\alpha^{(1)}_{[1,t-1]}$ to be the interval $[t-1]$ which would mean that the permutation is sum decomposable.
    So $\abs{\alpha^{(1)}}=1$ if it is not empty.
    In this case we must have $q<\alpha^{(1)}<r$, which yields the permutation 2413.\\

    Finally we show that all $n$-permutations avoiding $R_4$ are sum decomposable for $n\geq 5$.
    If $21[\beta_1,\beta_2]$ avoids $R_4$,
    then $\beta_1$ and $\beta_2$ must have lengths 1 or 2.
    If $2413[\beta_1,\beta_2, \beta_3, \beta_4]$ avoids $R_4$ ,
    then each $\beta_i$ must be of length exactly 1
    for all $i\in[4]$.
    The only sum indecomposable permutations avoiding $R_4$ are
    1, 21, 321, 312, 231 and 2413,
    so $\alpha$ in the theorem must be one of these.

\end{proof}

\begin{corollary}\label{cor:av10-an}
    Let $a_n:=\abs{Av_n(R_4)}$. Then
    \[a_n=\begin{cases}
        n!&\text{if }\,1\leq n\leq 3,\\
        12&\text{if }\, n=4,\\
        a_{n-1}+a_{n-2}+3a_{n-3}+a_{n-4}&\text{if }\, n\geq 5.
    \end{cases}\]
\end{corollary}
\begin{proof}
    The formula is clear for $n\leq 3$.
    Since $R_4$ represents a set of $12$ permutations, $Av_4(R_k)=4!-12=12$.
    For $n\geq 5$, recall
    that if $\pi=12[\alpha,\beta]$ for some permutations $\alpha$ and $\beta$
    where $\alpha$ is sum indecomposable,
    then $\alpha$ and $\beta$ are unique
    by Theorem \ref{thm:12,21}.
    The recursive formula then follows directly from Theorem \ref{thm:av10-an}.
\end{proof}

\subsection{Partial sums of signed displacements}

\begin{definition}
    The \textit{$j$th signed displacement} of an $n$-permutation $\pi$ for $j\in [n]$
    is defined as $\pi_j-j$.

    The \textit{$i$th partial sum of signed displacements} of an $n$-permutation $\pi$ for $i\in [n]$
    is defined as $\sum_{j=1}^i \pi_j-j$, and denoted $\pi_{\Sigma}^i$.
\end{definition}

We denote the set of $n$-permutations for which the partial sums of signed displacements do not exceed 2 by $\mathfrak{S}_n$.
For examples, refer to Figures \ref{conforming} and \ref{non-conforming}.

\begin{figure}[!htbp]
    \centering
    \begin{tabular}{c|cccc}
        $\pi_j$     & 3 & 1 & 4 & 2\\
        $j$         & 1 & 2 & 3 & 4\\
        $\pi_j-j$   & 2 & -1& 1 & -2\\\hline
        partial sum & 2 & 1 & 2 & 0
    \end{tabular}
    \caption{A permutation in $\mathfrak{S}_4$}
    \label{conforming}
    \bigskip
    \begin{tabular}{c|cccc}
        $\pi_j$     & 4 & 1 & 3 & 2\\
        $j$         & 1 & 2 & 3 & 4\\
        $\pi_j-j$   & 3 & -1& 0 & -2\\\hline
        partial sum & 3 & 2 & 2 & 0
    \end{tabular}
    \caption{A permutation that is not in $\mathfrak{S}_4$}
    \label{non-conforming}
\end{figure}

\begin{lemma}\label{lemma:av10-0}
    For an $n$-permutation $\pi$ and $k\in[n]$,
    we have $\pi_{\Sigma}^k=0$ if and only if
    $\pi_{[1,k]}$ is itself a $k$-permutation (without reducing it).
\end{lemma}
\begin{proof}

    We know that $\pi_j$ are positive and all distinct for all $j\in [k]$.
    So let $i_1,i_2,\dots, i_k$ be such that
    $1\leq \pi_{i_1}<\pi_{i_2}<\dots< \pi_{i_k}\leq n$.
    and let $r_j=\pi_{i_j}-j$ for all $j\in [k]$.
    Then $r_j$ is non-negative for all $j\in [k]$, and
    \begin{align*}
        \pi_{\Sigma}^k = 0
        \iff \sum_{j=1}^k (\pi_j -j)=0
        &\iff \sum_{j=1}^k \pi_{i} = \sum_{j=1}^k j\\
        &\iff \sum_{j=1}^k \pi_{i_j} = \sum_{j=1}^k j\\
        &\iff \sum_{j=1}^k (j+r_j) = \sum_{j=1}^k j\\
        &\iff \sum_{i=1}^k r_i =0.
    \end{align*}

    \noindent The last statement is true if and only if
    $r_j=0$ for all $j\in [k]$.
    That is, we must have $\{\pi_i\mid i\in [k]\}=[k]$,
    i.e. $\pi_{[1,k]}$ is a $k$-permutation.\\
\end{proof}

\noindent The proof of the following theorem was inspired by the proof on OEIS page on \href{http://oeis.org/A214663}{A214663}.

\begin{theorem}\label{thm:av10-bn}
    If $\pi\in \mathfrak{S}_n$ and $n\geq 5$,
    then $\pi$ must be of the form $12[\alpha,\beta]$,
    where \[\beta\in \{1, 21, 231, 312, 321, 3142\}\]
    and $\alpha$ is a conforming permutation $\alpha$ of length $n-\abs{\beta}$.
\end{theorem}
\begin{proof}
    For $n\geq 5$,
    observe that the number $n$ can only occur as one of the last 3 terms of a conforming $n$-permutation.
    Note also that the sum of two conforming permutations is conforming.
    Let $\pi$ be a conforming $n$-permutation.
    Then we have the following cases:
    \begin{itemize}
        \item Case 1: $\pi_n=n$.
        So $\pi_{n}-n=0$,
        and $\pi$ is conforming if and only if $\pi_{\Sigma}^{k}$
        does not exceed 2 for all $k\in [n-1]$.
        That is, $\pi$ is of the form $\alpha\oplus 1$ where $\alpha$
        is a conforming $(n-1)$-permutation.

        \item Case 2: $\pi_{n-1}=n$.
        So $\pi_{n-1}-(n-1)=1$,
        and $\pi$ is conforming only if $\pi_{\Sigma}^{n-2}=1$ or 0.
        We can further break down into cases:

        \begin{enumerate}[(a)]
            \item If $\pi_{\Sigma}^{n-2}=0$,
            then $\pi_{[1,n-2]}$ must itself be an $(n-2)$-permutation
            by Lemma \ref{lemma:av10-0}.
            Then we must have $\pi_n=n-1$ -
            that is, $\pi$ is of the form $\alpha\oplus 21$ where $\alpha$
            is a conforming $(n-2)$-permutation.

            \item If $\pi_{\Sigma}^{n-2}=1$ then $\pi_{\Sigma}^{n-1}=2$,
            so $\pi_n=n-2$.
            \begin{itemize}
                \item If $\pi_{\Sigma}^{n-3}=0$,
                then $\pi_{[1,n-3]}$ must itself be an $(n-3)$-permutation
                by Lemma \ref*{lemma:av10-0}.
                Since $\pi_{n-2}=n-1$ and $\pi_n=n-2$,
                so $\pi$ must be of the form $\alpha\oplus 231$
                where $\alpha$ is a conforming $(n-3)$-permutation.

                \item If $\pi_{\Sigma}^{n-3}=1=\pi_{\Sigma}^{n-2}$,
                then $\pi_{n-2}=n-2$,
                which is impossible since $\pi_n=n-2$.

                \item If $\pi_{\Sigma}^{n-3}=2$.
                then $\pi_{n-2}=n-3$.
                Since the partial sums of signed displacements of $\pi$ cannot exceed 2
                and $\pi_{[n-2,n]}=(n-3)\, n\,(n-2)$,
                we must have $\pi_{n-3}=n-1$.
                This means that $\pi_{\Sigma}^{n-4}=0$.
                Then $\pi_{[1,n-4]}$ must itself be an $(n-4)$-permutation
                by Lemma \ref*{lemma:av10-0},
                and $\pi$ is of the form $\alpha\oplus 3142$
                where $\alpha$ is a conforming $(n-4)$-permutation.
            \end{itemize}
        \end{enumerate}

        \item Case 3: $\pi_{n-2}=n$.
        So $\pi_{n-2}-(n-2)=2$,
        and $\pi$ is conforming if and only if the $(n-2)$th partial sum of signed displacements is 0.
        By Lemma \ref*{lemma:av10-0},
        the factor $\pi_{[1,n-3]}$ must be an $(n-3)$-permutation itself.
        It is easy to check that $321$ and $312$ are both conforming,
        so $\pi$ is of the form $\alpha\oplus 321$ or $\alpha\oplus 312$,
        where $\alpha$ is a conforming $(n-3)$-permutation.
    \end{itemize}

\end{proof}

\begin{corollary}\label{cor:av10-bn}
    Let $b_n:=\abs{\mathfrak{S}_n}$. Then
    \[b_n=\begin{cases}
        n!&\text{if }1\leq n\leq 3,\\
        12&\text{if }n=4,\\
        b_{n-1}+b_{n-2}+3b_{n-3}+b_{n-4}&\text{if }n\geq 5.
    \end{cases}\]
\end{corollary}
\begin{proof}
    It is easy to check that all $n$-permutations are conforming for $1\leq n\leq 3$.
    For $n=4$, there are exactly twelve conforming $n$-permutations,
    namely 1234, 1243, 1324, 1342, 1423, 1432, 2134, 2143, 2314, 3124, 3142 and 3214.
    For $n\geq 5$,
    the recursive formula for $b_n$ follows directly from Theorem \ref{thm:av10-bn}.
\end{proof}

\subsection{The bijection}

\begin{theorem}
    Define $\theta: Av_n(R_4)\rightarrow \mathfrak{S}_n$ as
    \[\theta(\pi) =
    \begin{cases}
        \pi                 &\text{ if }\abs{\pi}\in[4],\\
        \beta \oplus 3142 &\text{ if }\pi=2413\oplus \beta \text{ for some permutation }\beta,\text{ and }\\
        \beta \oplus \alpha &\text{ if } \pi=\alpha\oplus\beta \text{ for some permutations }\alpha\text{ and }\beta, \\
        &\text{ where }\alpha\neq 2413.
    \end{cases}\]
    Then $\theta$ is a bijection.
\end{theorem}

\begin{proof}
    We know that $\theta$ is indeed a function from $Av_n(R_4)$ to $\mathfrak{S}_n$
    by Theorems \ref{thm:av10-an} and \ref{thm:av10-bn}.
    It is clear from the definition that $\theta$ is an injection.
    Corollaries \ref{cor:av10-an} and \ref{cor:av10-bn} demonstrate that
    $\abs{\mathfrak{S}_n}=\abs{Av_n(R_4)}$ for all $n\geq 1$,
    so $\theta$ must be a bijection.
\end{proof}

\section{Simple permutations avoiding 2413, 3412 and 3421}\label{chap:fib}

Albert and Atkinson \cite{albert-atkinson} proved that
a permutation class with only finitely many simple permutations has a readily computable algebraic generating function
and has a finite basis.
So far, we have been dealing mostly with avoidance sets that contain only finitely many simple permutations.
In this chapter, we will show an example of a permutation class with a finite basis
that contains infinitely many simple permutations
as well as construct an algorithm that allows us to obtain the entire set recursively.

\begin{definition}
    Let $P$ be a pattern, a set of patterns, or a POP.
    Denote $Av_n^S(P)$ as the set of simple $n$-permutations avoiding $P$.
\end{definition}

\begin{theorem}\label{thm:fib}
    The set of simple permutations avoiding $P:=\{2413,\,3412,\,3421\}$ is enumerated by a translate of the well-known Fibonacci sequence.
    Specifically, if $n\geq 3$, then
    $\abs{Av_n^S(P)}=F(n-3)$
    where $F(0)=0$, $F(1)=1$ and $F(n)=F(n-1)+F(n-2)$ for all $n\geq 3$.
\end{theorem}

\subsection{Partitioning the simple permutations into 3 sets}

\noindent In order to enumerate the set of simple permutations avoid 2413, 3412 and 3421,
we will identify the types of permutations that appear in it.

\begin{lemma}\label{lemma:fib_structure}
    Let $\pi $ be a simple $n$-permutation that avoids $P$ for some $n\geq 4$.
    Then $\pi_{n-1}=n$ and $\pi_1=n-1$.
\end{lemma}

\begin{proof}
    We first prove that $\pi_{n-1}=n$.
    Let $j$ and $k$ be in $[n]$ such that $\pi_k=n$
    and \break $\pi_j := \max(\{\pi_1,\, \pi_2,\, \dots,\, \pi_{k-1]}\})$.
    Suppose there exists some $\ell>k$ such that $\pi_\ell<\pi_{j}$.
    Note that $\text{red}(\pi_j\, \pi_k\, \pi_{\ell_1}\, \pi_{\ell_2})$
    would be 3412 or 3421 if $\ell_1 < \ell_2$ and
    both $\ell_1$ and $\ell_2$ both satisfied conditions for $\ell$.
    So $\ell$ must be unique.
    If $\ell\neq n$ then $\pi_j<\pi_n$
    so $\pi_{j}\,\pi_k\,\pi_\ell\,\pi_n=\pi_{j}\, n\,\pi_\ell\,\pi_n$
    reduces to 2413.
    So $\ell=n$.
    We know that if $i>k$ and $i\neq \ell$ then $\pi_i>\pi_{j}$,
    which implies that $\pi_{[k,n-1]}$ is an interval.
    Since $\pi$ is simple and $k>1$, we must have $k=n-1$.  \\

    Next, we prove that $\pi_1=n-1$.
    Let $\pi$ be represented as the string of factors $\alpha\, (n-1)\, \beta\, n\, k$ where $k\in [n-2]$.
    We know that $\beta$ cannot be empty, since then $(n-1)n$ would be an nontrivial interval in $\pi$.
    Suppose $\alpha$ is not empty.
    Note that the subsequence $\hat{\alpha}\, \hat{\beta}\, k$
    must reduce to 123 or 132 or 231
    for all $\hat{\alpha}\in \alpha$ and $\hat{\beta}\in \beta$ by elimination
    (since $n-1$ is larger than the three points and $\hat{\alpha}\,(n-1)\, \hat{\beta}\, k$ cannot reduce to 2413, 3412 or 3421).
    The set of patterns $\{123,\, 132,\, 231\}$ is equivalent to the POP of size three where $1<2$,
    so the above observation implies that $\alpha < \beta$.
    If $\alpha > k$ then $\pi$ would be skew sum decomposable,
    while if $\alpha < k$ then $\pi$ would be sum decomposable,
    both of which contradict the simpleness of $\pi$.
    But this implies that $\beta > k$ which makes $(n-1)\, \beta\, n$ a nontrivial interval of $\pi$.
    So $\alpha$ must be empty, meaning that $\pi_1=n-1$.
\end{proof}

\begin{lemma}\label{lemma:fib_structure2}
    Let $\pi $ be a simple $n$-permutation that avoids $P$ where $n\geq 5$.
    Then we have exactly 3 cases:
    \begin{enumerate}
        \item $\pi_2=1$ and $\pi_3=n-2$,
        \item $\pi_2=n-3$ and $\pi_{n-2}=n-2$, or
        \item $\pi_2=1$, $\pi_3=n-3$ and $\pi_{n-2}=n-2$.
    \end{enumerate}
\end{lemma}

\begin{proof}
    We know from Lemma \ref{lemma:fib_structure} that $\pi_{n-1}=n$ and $\pi_1=n-1$.
    Write $\pi$ as the string of factors $(n-1)\,\gamma \, n\, k$ where $k$ has length 1.
    If $k=n-2$ then $\gamma$ is an interval of length $n-3\geq 2$. So $k\in [n-3]$ and $n-2\in \gamma$.
    Let $s:=\abs{\gamma}$ and $t\in [s]$ such that $\gamma_t=n-2$.\\

    Suppose $t<s$. We want to show that $t=2$ and $\gamma_1=1$, which will satisfy case a.
    If $t=1$ then $\pi_{[1,2]}=(n-1)\, (n-2)$ is a nontrivial interval, so $t>1$.
    So $t\in [2,s-1]$.
    Note that $\gamma_i\, (n-2)\, \gamma_j\, k$ reduces to 1423, 1432 or 2431 for all $1\leq i<t<j\leq s$
    by elimination since $n-2$ is the largest of the four points,
    and the subsequence cannot reduce to 2413, 3412 or 3421.
    This observation implies that $\gamma_{[1,t-1]}<\gamma_{[t+1,s-1]}$.
    So $\gamma_{[1,t-1]}<\gamma_{[t,s]}$.
    If $k<\gamma_{[t,s]}$ then $\gamma_{[t,s]}$ is a nontrivial interval.
    So $\gamma_{[1,t-1]} < k$ which means that $\gamma_{[1,t-1]}$ is an interval.
    By the simpleness of $\pi$ we have $t=2$.
    Since $\gamma_{[1,t-1]}=\gamma_1$ is smaller than all the other points, $\gamma_1=1$.\\

    Now suppose that $t=s$. That is, $\pi_{n-2}=n-2$.
    If $k=n-3$ then $\gamma$ is an interval of length 1.
    This is only possible if $\pi=41352$, which satisfies case a.
    So if $n\geq 6$ then $n-3$ is in $\gamma$ instead of $k$.
    Let $n\geq 6$ and $m\in [s]$ such that $\gamma_m=n-3$.
    If $m=1$ then $\pi$ satisfies case b.
    So suppose $m>1$.
    We know that $m\neq s-1$ since that would mean that $\gamma_{[s-1,s]}=(n-3)\,(n-2)$ is a nontrivial interval in $\pi$.
    So $m\in [2,s-2]$.
    Note that $\gamma_i\, (n-3)\, \gamma_j\, k$ reduces to 1423, 1432 or 2431 for all $1\leq i<m<j\leq s-1$,
    since $n-3$ is the largest of the four points
    and the subsequence cannot reduce to 2413, 3412 or 3421.
    This observation implies that $\gamma_{[1,m-1]} <\gamma_{[m+1,s-1]}$.
    So $\gamma_{[1,m-1]} <\gamma_{[m,s]} < n-1$.
    If $k<\gamma_{[m,s]}$ then $\gamma_{[m,s]}$ is a nontrivial interval.
    So $\gamma_{[1,m-1]} < k$ which means that $\gamma_{[1,m-1]}$ is an interval.
    By the simpleness of $\pi$ we have $m=2$.
    Since $\gamma_{[1,m-1]}=\gamma_1$ is smaller than all the other points, $\gamma_1=1$.\\
\end{proof}

By Lemma \ref{lemma:fib_structure} and \ref{lemma:fib_structure2},
we can partition the set of simple permutations of length $n$ avoiding $P$ by their initial terms.

\begin{definition}
    For $n\geq 4$, define $A_n$, $B_n$ and $C_n$ to be the set of simple $n$-permutations avoiding $P$
    beginning with $(n-1)\,1\,(n-2)$, $(n-1)\,(n-3)$ and $(n-1)\,1\,(n-3)$ respectively.\\
    Denote their cardinalities by $a_n$, $b_n$ and $c_n$.
\end{definition}

\begin{remark}
    The above definition holds for all $n\geq 4$ even if $A_n$, $B_n$ or $C_n$ are empty.
    Table \ref{table:fib-summary} summarizes the types of permutations in these sets due to Lemmas \ref{lemma:fib_structure} and \ref{lemma:fib_structure2}.
    Table \ref{table:fib_simples-sizes} gives sample values of $a_n$, $b_n$ and $c_n$ for $4\leq n\leq 8$.
\end{remark}

\begin{table}[!htbp]
    \centering
    \begin{tabular}{|c|cc|}
        \hline
        Set & \multicolumn{2}{c|}{Permutations in the set are of the form} \\\hline
        $A_n$ & $(n-1)\,1\,(n-2)\,\cdots\, n\, k$ & where $k\in [2,n-3]$ \\
        $B_n$ & 3142 or $(n-1)\,(n-3)\,\cdots\, (n-2)\, n\, k$ & where $k\in [2,n-4]$ \\
        $C_n$ & $(n-1)\,1\,(n-3)\,\cdots \,(n-2)\, n\, k$ & where $k\in [2,n-4]$ \\\hline
    \end{tabular}
    \label{table:fib-summary}
    \caption{Summary of the types of permutations in $Av_n^S(2413,3412,3421)$ for $n\geq 4$.
    The ellipses represent (possibly empty) factors of the permutations.}
\end{table}

\subsection{Recursive functions}

We will show that for all $n\geq 6$,
we can obtain simple permutations of length $n$ that avoid $P$
by adding points to smaller simple permutations that avoid $P$.

\begin{definition}
    For $n\geq 6$, let $S_n$ denote the set of permutations on $[n]$, and let
    \begin{align*}
        f_A:Av_{n-2}^S(P)
        \rightarrow S_{n},
        \quad
        f_B:Av_{n-2}^S(P)
        \rightarrow S_{n},
        \quad\text{and}\quad
        f_C:B_{n-1}
        \rightarrow S_{n},
    \end{align*}
    be functions defined as follows:
    \begin{align*}
        f_A(\pi_1 \pi_2\cdots \pi_{n-2})
        &:=(n-1)\,1\,(\pi_1+1)\,(\pi_2+1)\, \cdots\,(\pi_{n-3}+2)\,(\pi_{n-2}+1),\\
        f_B(\pi_1 \pi_2\cdots \pi_{n-2})
        &:=(n-1)\,\pi_1\,\pi_2\,\cdots\,\pi_{n-3}\, n\,\pi_{n-2},\quad\text{and}\\
        f_C(\pi_1 \pi_2\cdots \pi_{n-1})
        &:=(\pi_1+1)\,1\,(\pi_2+1)\,\cdots\,(\pi_{n-1}+1).
    \end{align*}
\end{definition}


\begin{definition}
    For $n\geq 6$, let
    \begin{align*}
        g_A: A_n\rightarrow S_{n-2},
        \quad
        g_B: B_n\rightarrow S_{n-2},
        \quad\text{and}\quad
        g_C: C_n\rightarrow S_{n-1},
    \end{align*}
    be functions defined as
    \begin{align*}
        g_A(\sigma)&:=(\sigma_3-1)\,(\sigma_4-1)\,\cdots\,(\sigma_{n-2}-1)\,(\sigma_{n-1}-2)\,(\sigma_n-1),\\
        g_B(\sigma)&:=\sigma_2\,\sigma_3\,\cdots\,\sigma_{n-2}\,\sigma_n,\\
        g_C(\sigma)&:=(\sigma_1-1)\, (\sigma_3-1)\, (\sigma_4-1)\, \cdots\, (\sigma_n-1).
    \end{align*}
    where $\sigma=\sigma_1\sigma_2\cdots \sigma_n\in Av_n^S(P)$.
    Note that $g_A(\sigma) = \text{red}(\sigma_{[3,n]})$,
    $g_B(\sigma) = \text{red}(\sigma_{[2,n-1]}\, \sigma_n)$
    and $g_C(\sigma) = \text{red}(\sigma_1\, \sigma_{[3,n]})$.
\end{definition}

\begin{lemma}\label{fib:f-A}
    If $\pi\in Av_{n-2}^S(P)$ and $n\geq 6$,
    then $f_A(\pi)$ is simple and avoids $P$.
\end{lemma}

\begin{proof}
    Recall that $f_A:Av_{n-2}^S(P) \rightarrow Av_{n}^S(P)$ and
    \[f_A(\pi_1 \,\pi_2\,\cdots \,\pi_{n-2})
    :=(n-1)\,1\,(\pi_1+1)\,(\pi_2+1)\,\cdots\,(\pi_{n-3}+2)\,(\pi_{n-2}+1).\]
    Suppose that $f_A(\pi)_{[i,j]}$ is an interval for some $i,\, j\in [n]$.
    The following cases show that all intervals in $f_A(\pi)$ are of length 0, 1 or $n$
    thereby proving that $f_A(\pi)$ is simple:

    \begin{itemize}
        \item Case 1: $i=1$. 
        Recall that $f_A(\pi)_1=n-1$ and $f_A(\pi)_2=1$,
        so if $j\geq 2$,
        then $f_A(\pi)_{[i,j]}$ must contain all the numbers in $[n-1]$
        and so must have length at least $n-1$.
        Observe that $f_A(\pi)_{n-1}=\pi_{n-3}+2=n$.
        So $f_A(\pi)_{[1,j]}$ must contain $[n]$, i.e. $j=n$.

        \item Case 2: $i=2$.
        Observe that $f_A(\pi)_2=1$,
        so $f_A(\pi)_{[2,j]}$ is the interval $[m]$ for some $m\in [n]$.
        If $j>2$, then $f_A(\pi)_{[3,j]}$ is the interval $[2,m]$.
        Recall that $f_A(\pi)_{[3,j]}=(\pi_{1}-1)\,(\pi_{2}-1)\,\cdots \,(\pi_{j-2}-1)$,
        so $\pi_{[1,j-2]}$ is an interval as well,
        and is a trivial interval only if $j\leq 3$.
        We know that $f_A(\pi)_{[2,3]}=1\, (n-3)$ is not an interval,
        so $j\leq 2$.

        \item Case 3: $i\geq 3$.
        If $j<n-1$ then $f_A(\pi)_{[i,j]}=(\pi_{i-2}-1)\,(\pi_{i-1}-1)\,\cdots \,(\pi_{j-2}-1)$.
        So $f_A(\pi)_{[i,j]}$ is an interval only if
        $\pi_{[i-2,j-2]}$ is an interval,
        which is true only if $j\leq i$ since $j-2<n-3$.
        Note that $f_A(\pi)_{n-1}=\pi_{n-3}+2=n$.
        So if $j\geq n-1$ then $f_A(\pi)_{[i,j]}$ contains the point $n$.
        This means that $f_A(\pi)_{[i,j]}$ can only be a nontrivial interval if it contains $\pi_1=n-1$ as well,
        which is impossible since $i\geq 3$.
        So $f_A(\pi)_{[i,j]}$ is an interval only if $j\leq i$.
    \end{itemize}

    Finally, we check that $f_A(\pi)$ avoids $P$.
    Suppose we have $1\leq i<j<k<\ell\leq n$ such that
    $f_A(\pi)_i\, f_A(\pi)_j\, f_A(\pi)_k\, f_A(\pi)_\ell$ reduces to a pattern in $P$.
    Since $\pi$ avoids $P$, such a subsequence
    must contain the numbers $n-1$ or $1$.
    However, the only term larger than $n-1$ is
    $n=\pi_{n-3}+2=f_A(\pi)_{n-1}$
    by Lemma \ref{lemma:fib_structure},
    while all patterns in $P$ have the largest term 4 as the third last number in the pattern,
    so it cannot contain the number $n-1$.
    On the other hand,
    1 is the second term of $f_A(\pi)$,
    but is the third or fourth term in the patterns in $P$.
    So it cannot contain $1$ either.
    So $f_A(\pi)$ avoids~$P$.
\end{proof}

\bigskip
\begin{lemma}\label{fib:f-B}
    If $\pi \in Av_{n-2}^S(P)$ and $n\geq 6$,
    then $f_B(\pi)$ is simple and avoids $P$.
\end{lemma}

\begin{proof}
    Recall that $f_B:Av_{n-2}^S(P) \rightarrow Av_{n}^S(P)$, and
    \[f_B(\pi_1 \pi_2\cdots \pi_{n-2})
    :=(n-1)\,\pi_1\,\pi_2\,\cdots\,\pi_{n-3}\, n\,\pi_{n-2}.\]

    Suppose that $f_B(\pi)_{[i,j]}$ is an interval for some $i,\, j\in [n]$.
    The following cases show that all intervals in $f_B(\pi)$ are of length 0, 1 or $n$
    thereby proving that $f_B(\pi)$ is simple:

    \begin{itemize}
        \item Case 1: $2\leq i,\, j\leq n-2$.
        Observe that $f_B(\pi)_{[i,j]}=\pi_{[i-1,j-1]}$
        is an interval only if $i\geq j$
        since $\pi$ is simple.
        So $i\geq j$.

        \item Case 2: $i=1\leq j\leq n-2$.
        Observe that $f_B(\pi)_{\ell}<f_B(\pi)_1=n-1$ for all $\ell\in [2,j]$.
        So the factor $f_B(\pi)_{[2,j]}=\pi_{[1,j-1]}$ must also be interval.
        Since $\pi$ is simple, and $j-1< n-2$,
        the factor $\pi_{[1,j-1]}$ is an interval only if $j\leq 2$.
        We know that $f_B(\pi)_{[1,2]}=(n-1)\pi_1=(n-1)\,(n-3)$ is not an interval,
        so we must have $j=1$.

        \item Case 3: $n-1\leq j$.
        Observe that $f_B(\pi)_{n-1}=n$.
        If $i<j$ then $f_B(\pi)_{[i,j]}$ must also contain $n-1=f_B(\pi)_1$, so $i=1$.
        So it is a factor of length at least $n-1$.
        But $f_B(\pi)_{[1,n-1]}$ is not an interval since it omits $f_B(\pi)_{n}=\pi_{n-2}$
        which is in $[2,n-2]$.
        So implies that $i=1$ and $j=n$.
        Otherwise $i\geq j$.
    \end{itemize}

    Finally, we prove that $f_B(\pi)$ avoids $P$.
    Suppose we have $1\leq i<j<k<\ell\leq n$ such that
    $f_B(\pi)_i\, f_B(\pi)_j\, f_B(\pi)_k\, f_B(\pi)_\ell$ reduces to a pattern in $P$.
    Then such a subsequence must contain the number $n-1=f_B(\pi)_1$ or $n=f_B(\pi)_{n-1}$
    since we know that $\pi$ avoids $P$.
    Clearly it cannot contain the number $n$ since
    it is the second last term in the permutation,
    but 4 never occurs as the last or second last term in patterns in 2413, 3412 or 3421.
    Then we must have $i=1$.
    But no patterns in $P$ start with 4,
    so $f_B(\pi)$ indeed avoids $P$.
\end{proof}

\break

\begin{lemma}\label{fib:f-C}
    If $\pi \in B_{n-1}$ and $n\geq 6$,
    then $f_C(\pi)$ is simple and avoids $P$.
\end{lemma}
\begin{proof}
    Recall that $f_C:B_{n-1} \rightarrow Av_{n}^S(P)$ and
    \[f_C(\pi_1\, \pi_2\, \cdots\, \pi_{n-1})
    :=(\pi_1+1)\,1\,(\pi_2+1)\,\cdots\,(\pi_{n-1}+1).\]

    It is easy to see that $f_C(\pi)$ avoids $P$.
    The only way that $f_C(\pi)$ could contain $P$ due to the addition of the point 1
    is if there are at least 2 terms preceding it.
    However, it is inserted in the second position,
    so $f_C(\pi)$ does not contain $P$.\\

    Suppose that $f_C(\pi)_{[i,j]}$ is an interval for some $i,\, j\in [n]$.
    The following cases show that all intervals in $f_B(\pi)$ are of length 0, 1 or $n$
    thereby proving that $f_C(\pi)$ is simple:

    \begin{itemize}
        \item Case 1: $i=1$.
        Suppose $j\geq 2$.
        Observe that $f_C(\pi)_1=\pi_1+1=n-1$ and $f_C(\pi)_2=1$.
        Then the interval $f_C(\pi)_{[i,j]}$
        must contain all the numbers in $[n-1]$.
        This includes $f_C(\pi)_n = \pi_{n-1} + 1$ which is in $[3,n-4]$ for $n\geq 6$.
        This is because $\pi_{n-1}$ cannot be $n-4$, $n-3$ or $n-2$
        by Lemmas \ref{lemma:fib_structure} and \ref{lemma:fib_structure2},
        and is not $1$ or $n-1$ since $\pi$ is simple.
        So $j=n$.

        \item Case 2: $i=2$.
        Observe that $f_C(\pi)_2=1$ and $f_C(\pi)_3=\pi_2+1=n-3$.
        If $j> 2$, then $f_C(\pi)_{[i,j]}$ is an interval only if it contains $[n-3]$.
        So the factor must have length at least $n-3$,
        meaning that $j\geq n-2$.
        Observe that $f_C(\pi)_{n-2}=\pi_{n-2}+1=n-1$ by Lemma \ref{lemma:fib_structure2},
        so $f_C(\pi)_{[i,j]}$ must contain $[n-1]$.
        However, this is not the case since it omits $f_C(\pi)_1=\pi_1+1$ which is not $n$.
        So $j\leq 2$.

        \item Case 3: $i\geq 3$.
        Observe that $f_C(\pi)_{[i,j]} =(\pi_{i-1}+1)\,(\pi_{i}+1)\,\cdots\,(\pi_{j-1}+1)$.
        Since $\pi$ is simple,
        the factor $\pi_{[i-1,j-1]}$
        is an interval only if $i \geq j$.
        So $i \geq j$.
    \end{itemize}
\end{proof}

\begin{lemma}\label{fib:g-A}
    If $\sigma=(n-1)\,1\,(n-2)\,\cdots\, n \, k\in A_{n}$ and $n\geq 6$,
    then $g_A(\sigma)$ is simple and avoids $P$.
\end{lemma}

\begin{proof}
    The mapping $g_A$ removes the first two entries and then reduces the result.
    Since $g_A(\sigma)$ is a reduced subsequence of $\sigma$,
    it avoids $P$.\\

    Suppose that $g_A(\sigma)$ is not simple.
    Then there exists $[a,b] \subsetneq [3,n]$ and $[c,d] \subseteq [2,n]$
    with $d \neq n-1$  and $b \geq a+2$ such that
    $\{\sigma_i: i\in [a,b]\}= [c,d]^0$ where $[c,d]^0 := [c,d]\setminus \{n-1\}$.\\

    If $d = n$ then $n-2 \in [c,d]^0$. Hence $a=3$ and $2 \in \{\sigma_i: i\in [a,b]\} = [c,d]^0$
    which implies that $k = \pi_n \in [c,d]^0$ as well.
    But we have excluded the case $[a,b]=[3,n]$
    since it is a trivial interval for $g_A(\sigma)$.
    This shows that $d \leq n-2$.
    But then $[c,d]^0=[c,d]$ and $\{\sigma_i: i\in [a,b]\}=[c,d]$ is an interval for $\sigma$.




\end{proof}

\begin{lemma}\label{fib:g-B}
    If $\sigma\in B_{n}$ and $n\geq 6$,
    then $g_B(\sigma)$ is simple and avoids $P$.
\end{lemma}

\begin{proof}
    The mapping $g_B$ removes the first entry and the second last entry of $\sigma$.
    Since $g_B(\sigma)$ is a reduced subsequence of the permutation $\sigma$,
    it avoids $P$.\\

    Suppose that $g_B(\pi)$ is not simple.
    Then there exists $[a,b] \subsetneq [2,n]$ and $[c,d] \subsetneq [1,n-2]$
    with $b \neq n-1$ and $c<d$ such that
    $\{\pi_{i}:i\in [a,b]^0\} = [c,d]$
    where $[a,b]^0 := [a,b] \setminus \{n-1\}$.
    If $b = n$ then $n-2 \in [a,b]^0$.
    Therefore $n-2 =\pi(n-2) \in [c,d]$
    and that implies that $(n-3) \in [c,d]$.
    Hence $a=2$ which means $[a,b]=[2,n]$.
    This contradiction implies that $b \leq n-2$.
    But then $[a,b]^0=[a,b]$ and $\pi_{[a,b]}=[c,d]$ is an interval for $\pi$.






\end{proof}

\begin{lemma}\label{fib:g-C}
    If $\sigma\in C_{n}$ and $n\geq 6$,
    then $g_C(\sigma)$ is simple and avoids $P$.
\end{lemma}

\begin{proof}
    The mapping $g_C$ removes the second entry (which is 1) and then reduces the result.
    It is not hard to see that $g_C(\sigma)$ is a reduced subsequence of the permutation $\sigma$,
    so it must avoid $P$.\\

    Suppose that $g_C(\pi)$ is not simple.
    Then there exists $[a,b] \subsetneq [1,n]$ and $[c,d] \subsetneq [2,n]$
    with $a \neq 2$ and $d \geq c+2$ such that
    $\{\pi_i: i\in [a,b]^0\} = [c,d]$
    where $[a,b]^0 := [a,b] \setminus \{2\}$.

    If $a = 1$ then $n-1 \in [c,d]$.
    Therefore $n-2$ or $n$ lies in $[c,d]$.
    Thus $b \geq n-3$ and this implies
    $2 \in \pi([a,b]^0) = [c,d]$.
    Therefore $k = \pi_n \in [c,d]$ and thus $b = n$,
    but we have excluded the case where $[a,b]=[1,n]$.
    This contradiction implies that $a \geq 3$.
    But then $[a,b]^0=[a,b]$ and $\pi_{[a,b]}=[c,d]$ is a nontrivial interval for $\pi$.



\end{proof}

\begin{table}[!htbp]
    \begin{center}
        \begin{tabular}{ | c | c | c | c | c | c |}
            \hline
            $n$ & \thead{\makecell{$A_n:=$ permutations\\that start with \\$(n-1)\,1\,(n-2)$}} & \thead{\makecell{$B_n:=$ permutations\\that start with \\$(n-1)\,(n-3)$}} & \thead{\makecell{$C_n:=$ permutations\\that start with \\$(n-1)\,1\,(n-3)$}} \\\hline
            4   &   -       &   3142    &   -       \\
            5   &   41352   &   -       &   -       \\
            6   &   514263  &   531462  &   -       \\
            7   &   6152473 &   6413572 & 6142573   \\
            8   & \makecell{71625384,\\71642583}    & \makecell{75142683,\\75314682}  &   71524683    \\
            \hline
        \end{tabular}
        \caption{Simple $n$-permutations avoiding $2413,\,3412,\,3421$ for $4\leq n\leq 8$}
        \label{table:fib_simples-perms}
    \end{center}
\end{table}

\begin{table}[!htbp]
    \begin{center}
        \begin{tabular}{ | c | c | c | c | c | c |}
            \hline
            $n$ & $Av_n^S(P)$ & $\abs{Av_n^S(P)}$ & $a_n$ & $b_n$ & $c_n$ \\\hline
            4   & 3142                                     & 1        & 0               & 1           & 0\\
            5   & 41352                                    & 1        & 1               & 0               & 0\\
            6   & \makecell{514263,\\531462}               & 2        & 1               & 1               & 0\\
            7   & \makecell{6152473, \\6142573, \\6413572} & 3        & 1               & 1               & 1\\
            8   & \makecell{71625384,\\71642583,\\71524683,\\75142683,\\75314682} & 5        & 2               & 2               & 1\\
            $k\geq 6$ & \makecell{$Av_k^S(P)$} & $F(k-3)$ & $F(k-5)$ & $F(k-5)$ & $F(k-6)$\\
            \hline
        \end{tabular}
        \caption{The number of simples avoiding $2413,\,3412,\,3421$, where $4\leq n\leq 8$}
        \label{table:fib_simples-sizes}
    \end{center}
\end{table}

\subsection{Proof of Theorem \ref{thm:fib}}

It is well known that there are no simple permutations of length $3$,
so it is obvious that $Av_3^S(P)=0=F(0)$.
For $n=4,\,5,\,6,\,7$, the values in Table \ref{table:fib_simples-sizes} are easy to verify.\\

We proceed to prove the enumeration for $n\geq 6$ via induction.
Recall that every simple $n$-permutation
avoiding $P$ begins with $n-1$
by Lemma \ref{lemma:fib_structure}.
So for all $n\geq 6$, if $\pi\in Av_{n-2}^S(P)$ then we have
\begin{align*}
    f_A(\pi_1 \pi_2\cdots \pi_{n-2})_{[1,3]}\quad
    &=\quad(n-1)\,1\,(\pi_1+1)\\
    &=\quad (n-1)\,1\,(n-2)
    \quad \text{and}\\
    f_B(\pi_1 \pi_2\cdots \pi_{n-2})_{[1,2]}\quad
    &=\quad (n-1) \, \pi_1\\
    &=\quad (n-1) \, (n-3),
\end{align*}

\noindent and for all $n\geq 6$ and $\pi\in B_{n-1}$ we have
\begin{align*}
    f_C(\pi_1 \pi_2\cdots \pi_{n-1})_{[1,3]}\quad
    &=\quad (n-1) \, 1 \, (\pi_2 +1)\\
    &=\quad (n-1) \, 1 \, (n-3).
\end{align*}

\noindent Therefore, from Lemmas \ref{fib:f-A}, \ref{fib:f-B} and \ref{fib:f-C},
\begin{align*}
    \pi\in Av_{n-2}^S(P) &\implies f_A(\pi)\in A_n
    \quad \text{and} \quad f_B(\pi)\in B_n,\\
    \text{and}\quad \pi\in B_{n-1} &\implies f_C(\pi)\in C_n,
\end{align*}

\noindent which means that $f_A(Av_{n-2}^S(P)) \subseteq A_n$,
$f_B(Av_{n-2}^S(P)) \subseteq B_n$ and $f_C(B_{n-1}^S(P)) \subseteq C_n$.

\noindent On the other hand,
Lemmas \ref{fib:g-A},\ref{fib:g-B} and \ref{fib:g-C} clearly show that
\begin{align*}
    \sigma\in A_{n} &\implies g_A(\sigma)\in Av_{n-2}^S(P)
    \quad \text{and} \quad \\
    \sigma\in B_{n} &\implies g_B(\sigma)\in Av_{n-2}^S(P)
    \quad \text{and} \quad \\
    \sigma\in C_{n} &\implies g_C(\sigma)\in B_{n-1}
    \quad \text{and} \quad
\end{align*}

\noindent which means that
$g_A(A_{n}) \subseteq Av_n^S(P)$,
$g_B(B_{n}) \subseteq Av_n^S(P)$ and
$g_C(C_{n}) \subseteq B_{n-1}$.
%
\noindent Therefore,
$g_A$, $g_B$ and $g_C$ is in fact the inverse of
$f_A$, $f_B$ and $f_C$ respectively.
Moreover, we claim that for all $n\geq 6$,
\[a_n = \abs{Av_{n-2}^S(P)},\quad
b_n = \abs{Av_{n-2}^S(P)},
\quad \text{and}\quad
c_n = b_{n-1}.\]

\noindent Suppose that for some $k\geq 6$,
\[a_\ell=F(\ell-5), \quad
b_\ell=F(\ell-5),
\quad \text{and}\quad
c_{\ell}=F(\ell-6) \]
for all $\ell\in\{4,5,\dots,k-1\}$.
Then from the statements above,
\begin{align*}
    a_{k}
    &=\abs{Av_{k-2}^S(P)}\\
    &=a_{k-2}+b_{k-2}+c_{k-2}\\
    &=F(k-7)+F(k-7)+F(k-8)\\
    &=F(k-7)+F(k-6)\\
    &=F(k-5), \\
    \text{so} \quad b_{k}
    =\abs{Av_{k-2}^S(P)}
    &=F(k-5)
    \quad \text{and}\quad
    c_k
    =b_{k-1}
    =F(k-6).
\end{align*}

\noindent So the claim is true by induction.
Therefore, for all $n\geq 6$,
\begin{align*}
    \abs{Av_n^S(P)}
    &=a_{n}+b_{n}+c_{n}\\
    &=F(n-5)+F(n-5)+F(n-6)\\
    &=F(n-5)+F(n-4)\\
    &=F(n-3).
\end{align*}

\noindent This concludes the proof of Theorem \ref{thm:fib}.
\qed

\section{Summary}\label{chap:conclusion}

In this paper, we elucidated connections between the avoidance sets of some POPs and other combinatorial objects by constructing explicit bijections between the relevant sets,
as a direct response to five of the 15 open questions posed by Gao and Kitaev \cite{gao-kitaev-2019}.
These bijections were derived primarily by analysing the simple permutations of the avoidance sets and how the rest of the set could be obtained from their inflations.
This was made possible by illustrating the permutation matrices as lattice matrices, which is a novel concept introduced in this paper.
The bijections constructed in this paper are a testament to the fundamental role that simple permutations play in the study of pattern-avoiding permutations.
It also demonstrates the intricate connections that avoidance sets of POPs have with many other combinatorial objects,
and provides a way to relate seemingly disparate combinatorial objects through their connections to the family of POPs.
We also enumerated the number of simple $n$-permutations avoiding the patterns 2413, 3412 and 3421 for all $n$,
giving a concrete example of an avoidance set with a finite basis and infinitely many simple permutations.\\

\section{Further Work}

The remaining ten questions posed by Gao and Kitaev \cite{gao-kitaev-2019} remain open.
Given the bijections we have constructed, it would be interesting to know whether they can be generalized further,
by studying generalizations of the POPs or of the combinatorial objects.
The following questions are natural extensions of the problem that was discussed in Section \ref{chap:levels}:
\begin{enumerate}[1.]
    \item Are there any combinatorial objects that have a natural bijective relationship with the avoidance set of $P_k$ for any $k\geq 5$?
    \item Is there a POP whose avoidance set is in bijection with the levels in compositions of ones, twos and threes?
    \item Are there any combinatorial objects that have a natural bijective relationship with the avoidance set
    of the POP with $k$ elements labelled $1,\, 2,\, \dots,\, k$ where $1>3>5$,
    or, more generally, with  $1>3>\cdots > 2i+1$ for some $i\geq 2$?
\end{enumerate}

The enumeration of $Av_n(R_k)$ for $k\geq 6$, $n\geq 1$, where $R_k$ is defined in Section \ref{chap:av10}, is an open question.
It could also be interesting to enumerate $\mathfrak{S}_{k,n}$,
which we define as the set of permutations whose partial sums of signed displacements do not exceed $k$, for all $k\geq 3$,
and check if there exist any $k$ and $\ell$ such that
$\abs{\mathfrak{S}_{k,n}}=\abs{Av_n(R_\ell)}$ for all $n\geq 1$.
Finally, Gao and Kitaev \cite{gao-kitaev-2019} observed that sequence $\abs{Av_{n-1}(R_4)}_{n\geq 2}$
corresponds to sequence \href{https://oeis.org/A232164}{A232164} as well.
The latter sequence counts the number of
Weyl group elements, not containing an $s_r$ factor, which contribute nonzero terms to Kostant's weight multiplicity formula when computing the multiplicity of the zero-weight in the adjoint representation for the Lie algebra of type $C$ and rank $n$.
Using our analysis on the set $Av(R_4)$,
one may be able to construct a natural bijection between these two sets more easily.\\

During our study of the simple permutations that avoid the patterns 2413, 3412 and 3421,
we discovered using the PermLab software that the addition of the pattern 2431 to the basis does not change the set of simple permutations for small $n$.
It can be proved that the simple permutations constructed by the recursive functions
to build the set $Av_n^S(2413,3412,3421)$ indeed avoid 2431.
This observation leads us to an interesting question: Which avoidance sets have the same set of simples?


\begin{thebibliography}{1}

\bibitem{permlab}
Michael Albert. PermLab. 2012. URL : http://www.cs.otago.ac.nz/PermLab/ 

\bibitem{albert-atkinson}
M.H. Albert and M.D. Atkinson. \emph{Simple Permutations and Pattern Restricted Permutations}.
Discrete Mathematics \textbf{300}. 1-3 (2005), pp. 1–15.

\bibitem{juggling-drops}
Joe Buhler, David Eisenbud, Ron Graham and Colin Wright. \emph{Juggling Drops and Descents}. The American Mathematical Monthly 101.6 (1994), pp. 507–519.
ISSN : 00029890, 19300972. URL : http://www.jstor.org/stable/2975316 .

\bibitem{chung-graham}
Fan Chung and Ron Graham. \emph{Primitive Juggling Sequences}. The American Mathematical Monthly 115.3 (2008), pp. 185–194.
DOI : http://www.jstor.com/stable/27642443 .

\bibitem{gao-kitaev-2019}
Alice Gao and Sergey Kitaev. \emph{On Partially Ordered Patterns of Length 4 and 5 in Permutations}.
The Electronic Journal Of Combinatorics 26.3 (2019).
DOI : https://doi.org/10.37236/8605 .

\bibitem{kitaev-textbook}
Sergey Kitaev. \emph{Patterns in Permutations and Words}. Springer Berlin Heidelberg, 2011.
DOI : https://doi.org/10.1007/978-3-642-17333-2 .

\bibitem{percus}
J.K. Percus. \emph{Combinatorial Methods}, Applied Mathematical Sciences \#4.
New York: Springer-Verlag, 1971. ISBN : 978-0-387-90027-8.

\bibitem{ivo}
Grant Pogosyan Ivo G Rosenberg Akihiro Nozaki Masahiro Miyakawa.
\emph{The number of orthogonal permutations}.
European Journal of Combinatorics 16.1 (1995), pp. 71–85.
DOI : https://doi.org/10.1016/0195-6698(95)90091-8

\bibitem{oeis}
N. J. A. Sloane. \emph{The Online Encyclopedia of Integer Sequences}. URL : https://oeis.org

\end{thebibliography}
\end{document}